%% file: limitsofproducts2.tex
\documentclass[options]{amsart}
\usepackage{latexsym,ifthen,amssymb}
\usepackage[toc,page,title,titletoc,header]{appendix}
\usepackage{enumerate}
\usepackage{graphicx}
\usepackage{bm}
\usepackage{subfigure}
\usepackage{float}
\usepackage{rotating}
\usepackage{epstopdf}
\usepackage{hyperref}

\usepackage{bm}

\usepackage[toc,page,title,titletoc,header]{appendix}
\usepackage{multirow}
 \usepackage{longtable}
 \usepackage{rotating}
 \usepackage{thmtools}
\usepackage{thm-restate}

\newtheorem{theorem}{Theorem}[section]
\newtheorem{lemma}[theorem]{Lemma}

\newtheorem{claim}[theorem]{Claim}

\newtheorem{proposition}[theorem]{Proposition}

\theoremstyle{definition}
\newtheorem{definition}[theorem]{Definition}

\theoremstyle{remark}

\newtheorem {question}[theorem]{Question}
\numberwithin{equation}{section}

\theoremstyle{noparens}
\newtheorem*{question*}{Question}
\newtheorem*{theorem*}{Theorem}
\newtheorem*{lemma*}{Lemma}

\newcommand{\uhr}{{\upharpoonright}}

\def\mcal{\mathcal}
\def\t{\tilde}

\def\h{\hat}
\def\v{\vec}

\def\msf{\mathsf}

\def\good{good}

\def\QX{Q(X,\cdot)}
\def\PX{P(X,\cdot)}

\def\exclude{exclude}
\def\mbk{\bm{k}}
\def\ch{\check}

\title{Coding power of product of partitions}

\author{Lu Liu}
\address{Department of Mathematics,
Central South University,
City Changsha, Hunan Province,
China. 410083}
\email{g.jiayi.liu@gmail.com}

\begin{document}
\begin{abstract}
 Given two combinatorial notions $\mathsf{P}_0$ and $\mathsf{P}_1$, can we encode $\mathsf{P}_0$ via $\mathsf{P}_1$. In this talk we address the question where $\mathsf{P}_0$ is  3-coloring of integers and $\mathsf{P}_1$ is  product of finitely many 2-colorings of integers.

       We firstly reduce the question to a lemma which asserts that certain $\Pi^0_1$ class of colorings admit two members violating a particular combinatorial constraint. Then we took a digression to see how complex does the class has to be so as to maintain the cross  constraint.

       We weaken the  two members in the lemma in certain way to address an open question
        of Cholak, Dzhafarov, Hirschfeldt and Patey, concerning a sort of Weihrauch degree of stable Ramsey's theorem for pairs. It turns out the resulted strengthen of the lemma is a basis theorem for $\Pi^0_1$ class with additional constraint. We look at several such variants of basis theorem, among them some are unknown.
             We end up by introducing some results and questions concerning product of infinitely many colorings.
\end{abstract}

\maketitle

\section{Introduction}

Given two combinatorial notions $\msf{P}_0,\msf{P}_1$,
does $\msf{P}_1$ encode $\msf{P}_1$. 
Depending on what ``encode" means, the question
appear in many fields of computer science.
For example, consider the error correcting code.
The encoding scheme is a function $f: k^n\rightarrow k^{\h n}$
such that for every $\rho_0\ne\rho_1\in k^n$,
the hamming distance between $f(\rho_0),f(\rho_1)$
is greater than certain value $m$.
In this case, given a piece of
information $\rho\in k^n$, a corrupted message $\h\rho\in k^{\h n}$
(which means the hamming distance between $\h \rho, f(\rho)$,
namely $|\{l:\h\rho(l)\ne f(\rho)(l)\}|$,
is below a certain value $m$), based on 
$\h\rho$ we can recover $\rho$ since $\rho$ is the unique
string in $k^n$ such that its $f$ value is within hamming distance $m$
to $\h\rho$.
In this example, 
an instance of problem $\msf{P}_0$ is a string $\rho\in k^n$,
and a solution to $\rho$ is $\rho$ itself;
an instance of problem $\msf{P}_1$ is a string $\h\rho\in k^{\h n}$,
and a solution to $\h\rho$ is a string $\t\rho\in k^{\h n}$
such that the hamming distance between $\t\rho,\h\rho$ is no greater 
than $m$.
By $\msf{P}_1$ encode $\msf{P}_0$, it means for every instance  $\rho$
of $\msf{P}_1$, there is an instance  of $\msf{P}_1$, namely $f(\rho)$,
such that every solution of $f(\rho)$ recovers $\rho$.

In computability theory, we define ``recover" by Turing computable.
Therefore, the resulted version of this encoding question
becomes: whether $\msf{P}_0$ is soc-reducible to 
$\msf{P}_1$. Where
  $\msf{P}_0$ is \emph{soc-reducible}
(strongly omniscent reducible \cite{monin2019pi}) to 
$\msf{P}_1$ iff for every $\msf{P}_0$ instance $I_0$,
there is a $\msf{P}_1$ instance  $I_1$ such that every solution to 
$I_1$ Turing computes a solution to $I_0$.
Let's see more example of instance solution framework. 
\begin{itemize}
\item For the problem $\msf{RT}_k^n$ (Ramsey's theorem for $n$ tuples with $k$ colors),
a $\msf{RT}_k^n$ instance is a coloring $C: [\omega]^n\rightarrow k$,
a solution is an infinite set $G\subseteq \omega$ monochromatic for $C$, i.e.,
$|C([G]^n)|=1$.

\item For the problem $\msf{DNR}_h$ ($h$-bounded $\msf{DNR}$ where $h$ is a function from 
$\omega$ to $\omega$),
a $\msf{DNR}_h$ instance is an infinite sequence of integers $X\in\omega^\omega$,
a solution to $X$ is a $Y\in h^\omega$ such that $Y(n)\ne X(n)$ for all $n\in\omega$.

\end{itemize}

\ \\

\noindent\textbf{Coding power of $(\msf{RT}_2^1)^{<\omega}$.}
In this paper, we address a particular encoding question——
whether $\msf{RT}_3^1$ is soc-reducible to $(\msf{RT}_2^1)^{<\omega}$.
Where an $(\msf{RT}_2^1)^{<\omega}$
instance is, for some $r\in\omega$,
an $r$-tuple of $2$-colorings $(C_0,\cdots,C_{r-1})\in (2^\omega)^r$
and a solution to $(C_0,\cdots,C_{r-1})$
is an $r$-tuple of infinite sets $(G_0,\cdots,G_{r-1})$
such that $G_s$ is monochromatic for $C_s$ for all $s<r$.

Firstly, we prove that $\msf{RT}_3^1$
is not soc-reducible to $(\msf{RT}_2^1)^{<\omega}$
in Theorem \ref{limitth0}.
Our  approach to sperate $\msf{RT}_3^1$ from  $(\msf{RT}_2^1)^{<\omega}$
appeals to many soc-reducibility questions
where the solution of a $\msf{P}_0,\msf{P}_1$
instance $I$ is an ``almost" $\Pi_1^0$ defined
combinatorial relation. Where ``almost" $\Pi_1^0$
means, taking $\msf{RT}_k^1$ as an  example:
for a $\msf{RT}_k^1$ instance $C$, 
the definition $G$ is solution to $C$
means,
(a) $G\subseteq C^{-1}(j)$ for some $j\in k$;
(b) $G$ is infinite.
Part (a) is a $\Pi_1^0$ relation; and part (b)
eliminate a small portion of $2^\omega$ from the solution space. 
This approach reduce Theorem \ref{limitth0} to a lemma
which asserts that for a $\Pi_1^0$ class $Q$
in the product space $\msf{P}_0\times \msf{P}_1$
(where for each member $(I_0,I_1)$ of $\msf{P}_0\times \msf{P}_1$,
$I_i$ is $\msf{P}_i$ instance), 
if  
the projection of $Q$ on its first component is sufficiently large,
then there are two members $(X^0,Y^0),(X^1,Y^1)$ of $Q$
such that, roughly speaking, $X^0,X^1$ has no common solution while 
$Y^0,Y^1$ has.
For Theorem \ref{limitth0},
this lemma is the following.
For two $\msf{RT}_k^1$ instances $C_0,C_1$, we say $C_0,C_1$ are \emph{almost disjoint}
if for every $j\in k$, $C_0^{-1}(j)\cap C_1^{-1}(j)$ is finite.
Let $Q\subseteq 3^\omega\times (2^\omega)^r$ be
a $\Pi_1^0$ class such that for every $X\in 3^\omega$,
there exists a $Y\in (2^\omega)^r$
such that $(X,Y)\in Q$ (in which case we say
$Q$ has \emph{full projection} on $3^\omega$).
\begin{lemma}\label{limitlem9}
There exist $(X^0,Y^0),(X^1,Y^1)\in Q$
such that:
suppose $Y^i=(Y^i_0,\cdots,Y^i_{r-1})$, we have
 $X^0,X^1$ are almost disjoint
and $Y^0_s,Y^1_s$ are not almost disjoint for all $s<r$.
\end{lemma}
The proof is given in section \ref{limitsec0} Lemma \ref{limitlem1}.
Lemma \ref{limitlem9} is interesting in its own and 
further rises two questions:
\begin{enumerate}
\item How complex does  $Q$ has to be in order to maintain the cross constraint,
i.e, \begin{align}\nonumber
&\text{for every }(X^0,Y^0),(X^1,Y^1)\in Q,
\text{ if }X^0,X^1\text{ are almost disjoint}, \\ \nonumber
&\text{ then }
Y^0_s,Y^1_s\text{ are almost disjoint for some }s<r.
\end{align}

\item When $Q$ is a $\Pi_1^0$  class,  how weak can the two members,
violating the constraint, be. For example,
can they be low?

\end{enumerate}

\ \\

\noindent\textbf{Complexity of the cross constraint.}
Note that when $r=1$,
the conclusion follows obviously and the reason
is ``finitary": there are three $3$-colorings
that are mutually almost disjoint while for every
three $2$-colorings, two of them are not  almost disjoint.

On the other hand,
 when $r>1$, for every finitely many $3$-colorings
$X^0,\cdots,X^{n-1}$,
there exist some $2$-colorings
$Y^0,\cdots,Y^{n-1}$
such that for every $m\ne m'<n$,
$X^m,X^{m'}$ being almost disjoint implies
that there exists $s<r$
 such that $Y^m_s,Y^{m'}_s$ are almost disjoint.
 Even if we let two players play this in a game,
 the party who wants to maintain the constraint  has
 a winning strategy. i.e.,
 The game is an infinite sequence $X^0,Y^0,X^1,Y^1,\cdots$
 where A presents $X^m$ then B presents $Y^m$;
 B wins iff for every $m\ne m'$,
 $X^m,X^{m'}$ being almost disjoint implies
that there exists $s<r$
 such that $Y^m_s,Y^{m'}_s$ are almost disjoint.
 It's easy to see that B has a winning strategy.
 But does B has a winning strategy without looking at
 the history of the game? i.e., Does there exists a function
 $f:3^\omega\rightarrow (2^\omega)^r$ such that
 for every $X^0,X^1$ being almost disjoint, there exists
 an $s<r$ such that the $s^{th}$ component of $f(X^0),f(X^1)$ are
 almost disjoint.
  Lemma \ref{limitlem1} says that such function $f$,
  if exists, can not be $\Delta_2^0$.
  We will further address
  these questions in section \ref{limitsec1}.


\ \\

\noindent\textbf{The weakness of the witnesses.}
Clearly, weakening the two members in Lemma \ref{limitlem9} 
give rise to constraint version basis theorem for $\Pi_1^0$ class.
For example, the assertion that  the two  members 
does not compute a given incomputable Turing degree, is 
the constraint version cone avoidance for $\Pi_1^0$ class;
the assertion that they (together) are low, is the constraint 
version low basis theorem.
Exploring the combinatorial nature of $Q$, we   prove the constraint
version cone avoidance as a demonstration how these basis theorems are proved.
We also introduce some similar constraint version basis ``theorem",
among which some are unknown.

Despite interesting in its own, the cross constraint version basis theorem
is also motivated by a question of Cholak, Dzhafarov, Hirschfeldt and Patey \cite{cholak2019some}.
Let $\msf{D}^2_k$ denote the problem that an instance 
is a coloring $C:[\omega]^2\rightarrow k$ 
that is stable, i.e., for every $n\in\omega$, 
there exists a $j\in k$ such that $C(\{n,m\}) = j$
for almost all $m\in\omega$;
a solution to $C$ is an infinite set $G$ monochromatic for $C$.
We say $\msf{P}_0$ is \emph{ computably reducible} 
to $\msf{P}_1$ iff: for every $\msf{P}_0$ instance $I_0$,
$I_0$ computes a $\msf{P}_1$ instance $I_1$ such that 
for every solution $G_1$ of $I_1$, $I_0\oplus G$ compute
a solution of $I_0$. 

\begin{question}\label{limitques0}
Is $\msf{D}^2_3$ computably reducible to $\msf{D}^2_2\times \msf{D}^2_2$.
\end{question}

By relativization, this question boils down to:
is there a $\Delta_2^0$ $\msf{RT}_3^1$ instance
$C$ such that for every $\Delta_2^0$ $\msf{RT}_2^1$ instances
$C_0,C_1$, there exists a solution $\v G$ to 
$(C_0,C_1)$ such that $\v G$ does not compute 
a solution to $C$.
In another word, this impose restriction on both the encoding instance
and the encoded instance. Therefore, it is neither
a strengthen or a weakening of the question 
whether $\msf{RT}_3^1$ is soc-reducible to $\msf{RT}_2^1\times \msf{RT}_2^1$.
However, we strength the result $\msf{RT}_3^1\nleq_{soc} (\msf{RT}_2^1)^{<\omega}$ by
showing, in Theorem \ref{limitth4}, that there is a $\Delta_2^0$
$\msf{RT}_3^1$ instance as a witness, answering 
question \ref{limitques0} in negative.
It turns out this strengthen can be proved by weakening 
the two members of Lemma \ref{limitlem9} in certain ways,
i.e., to preserve a sort of hyperimmune.

\ \\

\noindent\textbf{Related literature.}

We review some literature in computability theory
concerning weakness of a problem.
The are mainly two types of these questions,
the first one consider instance of certain complexity; while
the second one consider arbitrary instance.

In reverse math, 
it is established that there are five  axioms
where each is equivalent to many natural mathematical theorems
(see \cite{Simpson2009Subsystems} for an introduction on reverse math).
Among the five axioms, Recursive comprehension (or
$\msf{RCA}$ for short) is the weakest. For certain 
type of mathematical theorems
(those can be written as $\forall X\exists Y \psi(X,Y)$
where $X$ is seen as an instance and $Y$ is seen as a solution 
to $X$) $\msf{P}_0,\msf{P}_1$, to prove $\msf{P}_1$
does not imply $\msf{P}_0$ (over $\msf{RCA}$), it almost always boils down
to constructing weak solution of a computable instance of $\msf{P}_1$
(see \cite{Hirschfeldt2015Slicing} for recent development in reverse math).
For example, Seetapun and Slaman \cite{Seetapun1995strength} proved 
$\msf{RT}_2^2$ does not imply $\msf{RT}_2^3$ 
by showing that every computable instance of $\msf{RT}_2^2$
admits a solution that does not compute a given incomputable degree.

In computability randomness theory, there are  questions
concerning how to extract randomness by certain combinatorial notions $\msf{P}$.
Since not able to extract randomness is a sort of weakness,
these questions sometimes boils down to construct weak solution 
of $\msf{P}$ instance
of certain complexity.
For example, Greenberg and Miller \cite{greenberg2011diagonally}
proved that there exists a $\msf{DNR}_h$
that does not compute any $1$-random real.
Kjos-Hanssen asked whether every $1$-random real $X\in 2^\omega$
admit an infinite subset that does not compute any $1$-random real.
Thinking of ``subset" as a combinatorial notion where the solution
is an infinite subset of the instance, this is a question
concerning $1$-random instance of ``subset" problem.
The question is answered through a sequence of papers
\cite{Kjos-Hanssen2009Infinite}\cite{kjos2019extracting}.

Sometimes, studying the $\msf{P}$ instance of certain complexity 
translates to study an arbitrary instance of another problem.
For example, to study computable $\msf{RT}_2^2$ instance,
Seetapun and Slaman \cite{Seetapun1995strength} (and later Cholak, Jocksuch and Slaman 
\cite{Cholak2001strength}) applied so called CJS-decomposition to reduce 
$\msf{RT}_2^2$ to $\msf{SRT}_2^2$ and $\msf{COH}$.
It turns out, any computable $\msf{SRT}_2^2$ instance
is coded by a $\Delta_2^0$ $\msf{RT}_2^1$ instance.
Meanwhile, almost every avoidance results concerning $\Delta_2^0$
$\msf{RT}_2^1$ instance generalize to the strong avoidance version. 
Here we list a few of other examples studying an instance of arbitrary complexity
and show how these research are related to reverse math.

\begin{itemize}
\item 
A well known simple result is that fast-grow-function
(hence forth $\msf{FGF}$)
admit strong cone avoidance for non hyperarithmetic Turing degree.
i.e., Given a function  $f\in\omega^\omega$ there exists 
a function $g\geq f$ such that $g$ does not compute
a given non hyperarithmetic Turing degree.

\item A major question in reverse math
was whether $\msf{SRT}_2^2$ implies $\msf{COH}$.
This question boils down to constructing certain weak 
solution of a computable $\msf{SRT}_2^2$ instance.
In \cite{Dzhafarov2016Ramseys}, Dzhafarov, Patey, Solomon and Westrick considers arbitrary instance of $\msf{SRT}_2^2$,
and studied whether $\msf{COH}$ is soc-reducible to $\msf{SRT}_2^2$.
 They proved that 
$\msf{RT}_2^1$ is not soc-reducible to $\msf{FGF}$;
actually they proved that $\msf{RT}_{k+1}^1$ is not 
soc-reducible to $\msf{FGF}\times \msf{RT}_k^1$,
thus $\msf{SRT}_2^2$ is not soc-reducible to $\msf{COH}$
since $\msf{SRT}_2^2$ is soc-equivalent to
$\msf{RT}_2^1\times\msf{FGF}$.

Chong, Slaman and Yang employ the nonstandard model
answering the  $\msf{SRT}_2^2$ vs $\msf{COH}$ question 
in negative.
Later, Monin and Patey \cite{monin2019weakness} study the jump control of 
an arbitrary $\msf{RT}_2^1$ instance, this approach finally
solve the
 standard model version of the question.
 More recently, Monin and Patey proved strong hyperarithmetic
 and arithmetic
 avoidance of $\msf{RT}_2^1$. Their method also
 answers some questions of Wang \cite{Wang2014Cohesive}
 and generalize the low$_2$ construction
 in \cite{Cholak2001strength}.

\item An old question of Sacks asked whether 
there is a solution to a the universal $\msf{DNR}$ instance
of minimal degree, which is answered in affirmative 
by Kumabe \cite{Kumabe2009fixed} using bushy tree method.
Later, people start to wonder whether there is a real of positive 
hausdorff dimension that is of minimal degree.
These reals are coded by
$\msf{DNR}_h$ for some    computable non decreasing unbounded function $h$.
In \cite{khan2017forcing}\cite{liu2019avoid},  it is shown that 
for an arbitrary $\msf{DNR}_h$ instance admit minimal degree solution
if $h$ is fast enough.
Hopefully, reducing the fast growing condition of $h$ could finally solve the question.

\item A major question in reverse math was whether 
$\msf{RT}_2^2$ implies 
$\msf{WKL}$, where $\msf{WKL}$ is the weak K$\ddot{o}$nig Lemma. In \cite{Dzhafarov2009Ramseys},
Dzhafarov and Jockusch proved strong cone avoidance of $\msf{RT}_2^1$.
This is later generalized by author \cite{Liu2012RT22} to strong avoidance of PA degree, answering
the question negatively.

\item Downey, Greenberg, Jockusch and Milans \cite{downey2011binary} proves that
$\msf{DNR}_2$ is not  soc-reducible
to $\msf{DNR}_3$ in a uniform manner, i.e., given
a Turing function $\Psi$, there always exists
a $\msf{DNR}_2$ instance $X$ such that for every $\msf{DNR}_3$
instance $\h X$, there is a solution to $\h X$ that  does not compute
 a solution to $X$.

\end{itemize}

\ \\

\noindent\textbf{Organization.}
In section \ref{limitsec0} we prove Lemma \ref{limitlem9}
and demonstrate how
$\msf{RT}_3^1\nleq_{soc}(\msf{RT}_2^1)^{<\omega}$
can be proved using Lemma \ref{limitlem9}.
Section \ref{limitsec1}
discuss how complex does the set $Q$ in Lemma \ref{limitlem9}
need to be in order to satisfy the cross constraint.
In section \ref{limitsec2}, we presents some 
constraint version basis theorem; among them is the preservation 
of $\Gamma$-hyperimmune of a $3$-coloring (see section \ref{limitsubsec2}
for a definition). 
In section \ref{limitsec3}, we introduce some results
and questions on product on infinitely many colorings.


\ \\

\noindent\textbf{Notations.}
For a sequence of poset $(W_0,<_{p_0}),\cdots,(W_{n-1},<_{p_{n-1}})$,
and $(x_0,\cdots,x_{n-1})$ ,
$(y_0,\cdots,y_{n-1})\in W_0\times\cdots\times W_{n-1}$,
we write $(y_0,\cdots,y_{n-1})\leq_p$
$(<_p\text{ respectively})$
$(x_0,\cdots,x_{n-1})$
if $y_m\leq_{p_m}
(<_{p_m} \text{ respectively}) x_m$ for all $m<n$.

We use $\sigma,\tau,\rho,\xi$ to denote strings of finite length.
 For a string $\rho$, we let $[\rho]^{\preceq}=\{\sigma: \sigma\succeq \rho\}$;
similarly, for a set $S$ of strings, let $[S]^{\preceq} = \{\sigma: \sigma\succeq \rho\text{ for some }\rho\in S\}$.
For a tree $T$, let $[T]$ denote the set of infinite paths on $T$;
for a string $\rho$,
 let $[\rho]= \{X\in \omega^\omega:X\succeq \rho\}$;
 for a finite set $V$ of strings, let $[V]=\{[\rho]:\rho\in V\}$.
 For a string $\rho$, we use $|\rho|$ to denote the length of $\rho$.
 For a set $S$ of strings, let $\ell(S)$
 denote the set of leaves of $S$ (i.e., the set of strings in $S$
 having no proper extension in $S$ when $S\ne\emptyset$)
 and we rules $\ell(\emptyset) = \bot$.

Two $k$-colorings $X_0,X_1\in k^\omega$
are \emph{almost disjoint on }$Z\subseteq \omega$ if for every
$j\in k$, $X_0^{-1}(j)\cap X_1^{-1}(j)\cap Z$
is finite; when $Z=\omega$, we simply say
$X_0,X_1$ are almost disjoint.

\section{Weakness of a product of partitions}
\label{limitsec0}
The main objective in this section is to address the 
question whether $\msf{RT}_3^1$ is soc-reducible to 
$(\msf{RT}_2^1)^r$.
Let $(\msf{RT}_2^1)^r$
denote the problem whose instance is an $r$-tuple $(C_0,\cdots, C_{r-1})$ of $2$-colorings,
a solution to $(C_0,\cdots,C_{r-1})$ is an $r$-tuple of infinite sets of integers
$(G_0,\cdots,G_{r-1})$ such that $G_s$ is monochromatic for $C_s$ for all $s<r$.
Let $(\msf{RT}_2^1)^{<\omega}$ denote 
$\bigcup_r (\msf{RT}_2^1)^r$.

\begin{theorem}\label{limitth0}
We have $\msf{RT}_3^1\nleq_{soc}
(\msf{RT}_2^1)^{<\omega}$. Moreover, there exists a $\emptyset^{(\omega)}$-computable
 $\msf{RT}_3^1$ instance  
witnessing the conclusion.

\end{theorem}

The rest of this section will prove
Theorem \ref{limitth0}.
Let's firstly briefly discuss the framework of its proof.
First, we fix a somewhat complex $\msf{P}_0$ instance $I_0$ and an arbitrary
 $\msf{P}_1$ instance $I_1$. The goal is to construct a solution of 
 $I_1$ that does not compute any solution of 
 $I_0$.
 The general approach is to construct  a sequence of
conditions $d_0,d_1,\cdots$:
\begin{itemize}
\item where each condition
is essentially a closed set of
candidates of the weak solution we construct,

\item each requirement is forced by some condition
(meaning every member of the condition satisfy the requirement).
\end{itemize}
Then we 
take the common element $G$ of $d_0,d_1,\cdots$
(which exists by compactness),
so it satisfies all requirements. Thus $G$ is the desired 
weak solution of $I_1$
 not computing any solution of $I_0$. 

The main point in this approach is how to extend a condition 
in order to force a given requirement. 
Usually, this turns out to be a uniform encoding question.
i.e., whether for every instance of the encoded problem, 
there  is an instance of the encoding problem so that every solution 
of the encoding instance compute, \emph{via a given algorithm}, a solution
of the given encoded instance.
In  particular, in theorem \ref{limitth0}, the uniform encoding 
question is exactly the uniform version of the original encoding question
(which might not be the case in general).
i.e., roughly speaking, it boils down to the following:
let $C$ be the complex $\msf{RT}_3^1$ instance; 
\begin{align}\label{limiteq19}
&\text{ given a tuple of Turing functional }\{\Psi_{\mbk}\}_{\mbk\in 2^r},
\ \ \text{ a tuple of colors }\{j_{\mbk}\}_{\mbk\in 2^r}\subseteq 3,
\\ \nonumber
&\text{   a }(\msf{RT}_2^1)^r \text{ instance }(C_0,\cdots,C_{r-1}),
\\ \nonumber
 &\text{ there is a solution }\v G \text{ in color }\mbk\text{ of }(C_0,\cdots,C_{r-1})\text{ for some }\mbk
 \text{ such that }
 \\ \nonumber
 &\ \ \Psi^{\v G}\text{ is not a solution of }C\text{ in color }j_{\mbk}.
\end{align}

If $C$ is not encoded via $\{\Psi_{\mbk}\}_{\mbk\in 2^r},\{j_{\mbk}\}_{\mbk\in 2^r}$
in a fashion as (\ref{limiteq19}), then we can finitely extend the initial segment of a 
condition to force the Turing functional $\Psi_{\mbk}$ to 
violate $C$ deterministically.

On the other hand, if $C$ is encoded via 
$\{\Psi_{\mbk}\}_{\mbk\in 2^r},\{j_{\mbk}\}_{\mbk\in 2^r}$ in that way,
we observe the behavior of 
$\Psi_{\mbk}$ and  look at all $3$-colorings $\t C$ encoded 
in that way. 
It turns out that the encoded $\t C$ forms a $\Pi_1^0$ class.
Since the very complex $C$ is in it, this class contains
a ``lot" of $3$-colorings. Now
we look at the set of pairs $(\t C,\h C)$
where  $\t C$ is uniformly encoded by $\h C$ 
via
$\{\Psi_{\mbk}\}_{\mbk\in 2^r},\{j_{\mbk}\}_{\mbk\in 2^r}$ similarly as (\ref{limiteq19}).
We prove in Lemma \ref{limitlem1} that there are $(\t C^0, \h C^0),(\t C^1, \h C^1)$ 
such that 
\begin{itemize}
\item $\t C^0,\t C^1$ has no common solution 
in any color $j\in 3$ and
\item  $\h C^0,\h C^1$ do have a common solution 
in some color $\mbk^*\in 2^r$.
\end{itemize}
Thus, let $\v G$ be that common solution of $\h C^0,\h C^1$.
We have $\Psi_{\mbk^*}^{\v G}\subseteq^* (\t C^0)^{-1}(j_{\mbk^*})\cap (\t C^1)^{-1}(j_{\mbk^*})$,
therefore it is a finite set.

More specifically, what we need is the following lemma.
Let $Q\subseteq 3^\omega\times (2^\omega)^r$ be
a $\Pi_1^0$ class that has full projection on $3^\omega$.
\begin{lemma}\label{limitlem1}
There exist $(X^0,Y^0),(X^1,Y^1)\in Q$
such that:
suppose $Y^i=(Y^i_0,\cdots,Y^i_{r-1})$, we have
 $X^0,X^1$ are almost disjoint
and $Y^0_s,Y^1_s$ are not almost disjoint for all $s<r$.
Moreover, $(X^0,Y^0)\oplus (X^1,Y^1)\leq_T \emptyset'$.
\end{lemma}
\begin{proof}
The approach is to build two sequences
of pairs $(\rho^i_0,\sigma^i_0)\prec (\rho^i_1,\sigma^i_1)\prec\cdots$,
$i\in 2$ such that the colorings $X^i=\cup_t \rho^i_t$,
$Y^i = \cup_t \sigma^i_t = (Y^i_0,\cdots, Y^i_{r-1})$
are the desired witness.

For $\rho\in 3^{<\omega}$,
$\sigma\in 2^{<\omega}$,
 $(\rho,\sigma)$ is \emph{\good}
iff
for every $X\in [\rho]$,
 $\QX\cap [\sigma]\ne\emptyset$
 where $\QX=\{Y\in (2^\omega)^r: (X,Y)\in Q\}$.
 It's easy to see that
 \begin{align}\label{limiteq1}
 &\text{ for every good pair }
 (\rho,\sigma),\text{ every }\h\rho\succeq \rho,
 \\ \nonumber
 &\text{ there exists a good pair }
 (\t\rho,\t\sigma)\succ(\h\rho,\sigma).
 \end{align}

A \emph{condition} in this lemma
is a $2$-tuple $((\rho^0,\sigma^0), (\rho^1,\sigma^1))$
such that  $(\rho^i,\sigma^i)\in 3^{<\omega}\times 2^{<\omega}$ is good for all $i\in 2$.
We will try to build the desired two sequences by
$\emptyset'$-computing a sequence of condition which
fulfill all necessary requirements.
More specifically, for each condition
$((\rho^0_t,\sigma^0_t), (\rho^1_t,\sigma^1_t))$
in the sequence, try to extend it to
 $((\rho^0_{t+1},\sigma^0_{t+1}), (\rho^1_{t+1},\sigma^1_{t+1}))$
 so that 
 \begin{itemize}
 \item $(\rho^0_{t+1},\rho^1_{t+1})$ does not \emph{positively progress
 on any color} $j\in 3$ \emph{ compared to }$(\rho^0_t,\rho^1_t)$,
 $$\text{i.e., }(\rho^0_{t+1})^{-1}(j)\cap (\rho^1_{t+1})^{-1}(j)
 = (\rho^0_{t})^{-1}(j)\cap (\rho^1_{t})^{-1}(j)
 \text{ for all }j\in 3;$$
\item while for each $s<r$, for some $k\in 2$, $(\sigma^0_{t+1},\sigma^1_{t+1})$
\emph{positively progress on color $k$ on the $s^{th}$ component
compared to }$(\sigma^0_t,\sigma^1_t)$,
$$\text{ i.e., }(\sigma^0_{t+1,s})^{-1}(k)\cap (\sigma^1_{t+1,s})^{-1}(k)
\supsetneq (\sigma^0_{t,s})^{-1}(k)\cap (\sigma^1_{t,s})^{-1}(k).$$
Where $\sigma^i_t=(\sigma^i_{t,0},\cdots,\sigma^i_{t,r-1})$.
\end{itemize}
If we can keep positively progress on each component, then we are done.
We will show in Claim \ref{limitclaim0} that this is indeed
 the case.

A condition $((\rho^0,\sigma^0), (\rho^1,\sigma^1))$
\emph{\exclude} component $s$ if it cannot  be extended
to make positive progress on  the $s^{th}$ component.
More specifically, for every condition
$((\h{\rho}^0,\h{\sigma}^0), (\h{\rho}^1,\h{\sigma}^1))$
extending $((\rho^0,\sigma^0), (\rho^1,\sigma^1))$,
if $(\h{\rho}^0,\h{\rho}^1)$ does not positively progress on
color any color $j\in 3$
  compared to $(\rho^0,\rho^1)$, then $(\h{\sigma}^0,\h{\sigma}^1)$
does not positively progress on any color
on the $s^{th}$ component compared to $(\sigma^0,\sigma^1)$.

\begin{claim}\label{limitclaim0}
If  condition $((\rho^0,\sigma^0), (\rho^1,\sigma^1))$
\exclude\ component $s$,
then there exists
two computable $Y^0,Y^1\in 2^\omega$
such that for every $i\in 2$, every \good\ pair
$(\h{\rho}^i,\h{\sigma}^i)\succeq
(\rho^i,\sigma^i)$,
we have $\h{\sigma}^i_s\prec Y^i$
 $($in which case we say $(\rho^i,\sigma^i)$
 \emph{lock} the $s^{th}$ component$)$.
\end{claim}
\begin{proof}
It suffices to prove that for every sufficiently large
$n\in \omega$,
there exists $k_0,k_1\in 2$ (depending on $n$)
such that for
every
 condition
$((\h{\rho}^0,\h{\sigma}^0), (\h{\rho}^1,\h{\sigma}^1))$
extending $((\rho^0,\sigma^0), (\rho^1,\sigma^1))$,
$\h{\sigma}^i_s(n)=k_i$ whenever $\h{\sigma}^i_s(n)$
is defined.
Suppose otherwise.
Then there exists an $i\in 2$ (say $i=0$) and
two \good\ pairs $(\h{\rho},\h{\sigma})$,
$(\t{\rho},\t{\sigma})$ extending
$ (\rho^0,\sigma^0)$
such that $\h{\sigma}_s(n)\ne\t{\sigma}_s(n)$
are defined.
Now extends $(\rho^1,\sigma^1) $ to a \good\ pair
$(\rho',\sigma')$ where $\sigma'_s(n)$ is defined
and neither of  $(\rho',\h{\rho}),(\rho',\t{\rho})$
 positively progress on any color $j\in 3$
compared to $(\rho^1,\rho^0)$.
But clearly, either
$(\sigma',\h{\sigma})$ or $(\sigma',\t{\sigma})$
will positively progress on some color $k\in 2$ on the
$s^{th}$ component compared to $(\sigma^1,\sigma^0)$
(depending on $\sigma'_s(n)=\h{\sigma}_s(n)$
or $\sigma'_s(n)=\t{\sigma}_s(n)$).
A contradiction to \exclude\ component $s$.

\end{proof}

Let $I\subseteq r$ be a maximal set such that
there exists a \good\ pair $(\rho,\sigma)$
lock $s^{th}$ component for all $s\in I$.
Starting with the condition
$((\rho,\sigma),(\rho,\sigma))$, it is clear to see that
for any condition $((\rho^0,\sigma^0), (\rho^1,\sigma^1))$
extending $((\rho,\sigma),(\rho,\sigma))$
with $(\rho^0,\rho^1)$  not positively progress on any
color $j\in 3$ compared to $(\rho,\rho)$, $((\rho^0,\sigma^0), (\rho^1,\sigma^1))$
does not \exclude\ any component.
For example, if $I=r$, then there even exists
computable $X^0,X^1\in 3^\omega, Y\in 2^\omega$
such that $(X^0,Y),(X^1,Y)\in Q$ are as desired, namely
$X^0,X^1$ are almost disjoint.
Thus we will be able to keep making positive progress on each
component.
Thus we are done.

\end{proof}

Now we are ready to prove Theorem \ref{limitth0}.
\begin{definition}[Hyperimmune] 
A $\msf{RT}_k^1$ instance
$C$ is \emph{hyperimmune} relative to $D$ if
for every $D$-computable array of
 $k$-tuple of mutually disjoint finite
sets $\{(F_{n,0},\cdots,F_{n,k-1})\}_{s\in\omega}$
with $\bigcup_{j<k} F_{n,j}>n$ for each $n$,
there exists an $n^*$ such that
$F_{n^*,j}\subseteq C^{-1}(j)$ for all $j<k$.

\end{definition}

\begin{proof}[Proof of Theorem \ref{limitth0}]
Let $C$ be a $\emptyset^{(\omega)}$-computable $\msf{RT}_3^1$ instance hyperimmune
 relative to any arithmetic degree;
 let $r\in \omega$ and let $(C_0,\cdots,C_{r-1})$
 be a $ (\msf{RT}_2^1)^r$
instance.

The \emph{condition} we use in this theorem
is a tuple $(\{\v{F}_{\mbk}\}_{\mbk\in 2^r}, \v{X})$
such that
\begin{itemize}
\item Each $\v{F}_{\mbk}=(F_{\mbk,0},\cdots,F_{\mbk,r-1})$
is an $r$-tuple of finite set in color $\mbk$ of $(C_0,\cdots,C_{r-1})$.
i.e., $F_{\mbk,s}\subseteq C_s^{-1}(\mbk(s))$ for all $s<r$.

\item $\v{X}=(X_0,\cdots,X_{r-1})$ is arithmetic; each
$X_s\subseteq \omega$ is infinite and $X_s>F_{\mbk,s}$
for all $\mbk\in 2^r,s<r$.

\end{itemize}
As in the Mathias forcing, 
let $(\v F_{\mbk},\v X)$ denote the
collection
\begin{align}\nonumber
\{(G_0',\cdots,G_{r-1}'):
F_{\mbk,s}\subseteq G_s'\wedge G_s'\subseteq F_{\mbk,s}\cup X_s
\text{ for all }s<r\}.
\end{align}

A condition $(\{\v{F}_{\mbk}\}_{\mbk\in 2^r}, \v{X})$ is seen as a collection of candidates of
the solution $(G_0,\cdots,G_{r-1})$ we construct. Namely:
$\bigcup_{\mbk\in 2^r}(\v F_{\mbk},\v X)$.

A condition $(\{\v{\h{F}}_{\mbk}\}_{\mbk\in 2^r}, \v{\h{X}})$
\emph{extends} another condition $(\{\v{F}_{\mbk}\}_{\mbk\in 2^r}, \v{X})$
(written as
$(\{\v{\h{F}}_{\mbk}\}_{\mbk\in 2^r}, \v{\h{X}})$
$\subseteq
(\{\v{F}_{\mbk}\}_{\mbk\in 2^r}, \v{X})$)
if the collection represented by $(\{\v{\h{F}}_{\mbk}\}_{\mbk\in 2^r}, \v{\h{X}})$
is a subset of that of $(\{\v{F}_{\mbk}\}_{\mbk\in 2^r}, \v{X})$. Or in another word,
$F_{\mbk,s}\subseteq \h{F}_{\mbk,s}\subseteq F_{\mbk,s}\cup X_s, \h{X}_s\subseteq X_s $ for all
$s<r,\mbk\in 2^r$.

Given a tuple of requirements $\{\mcal{R}_{\mbk}\}_{\mbk\in 2^r}$,
a condition $(\{\v{F}_{\mbk}\}_{\mbk\in 2^r}, \v{X})$
\emph{forces} $\bigvee_{\mbk\in 2^r}\mcal{R}_{\mbk}$
iff there \emph{exists} a $\mbk\in 2^r$ such that for
 every $(G_0,\cdots,G_{r-1})\in (\v{F}_{\mbk},\v{X})$, we have
$(G_0,\cdots,G_{r-1})$ satisfies $\mcal{R}_{\mbk}$.

For every Turing functional $\Psi$, every $j\in 3$,
let $\mcal{R}_\Psi^j$ denote the requirement:
$$\Psi^{(G_0,\cdots,G_{r-1})}
\text{ is not an infinite subset of } C^{-1}(j).$$
Fix a tuple of requirements  $\{\mcal{R}_{\mbk}\}_{\mbk\in 2^r}$,
where for some Turing functional
$\Psi_{\mbk}$, some $j_{\mbk}\in 3$,
$\mcal{R}_{\mbk}$ is $\mcal{R}_{\Psi_{\mbk}}^{j_{\mbk}}$;
fix a condition $(\{\v{F}_{\mbk}\}_{\mbk\in 2^r}, \v{X})$.
We will show that
there exists an extension of $(\{\v{F}_{\mbk}\}_{\mbk\in 2^r}, \v{X})$
that forces  $\bigvee_{\mbk\in 2^r}\mcal{R}_{\mbk}$.
In this way, we will have a sequence of conditions
$d_0\supseteq d_1\supseteq\cdots \supseteq d_t=(\{\v{F}_{\mbk}^t\}_{\mbk\in 2^r}, \v{X}^t)\supseteq\cdots$
such that  for every $2^r$-tuple of Turing functionals
$\{\Psi_{\mbk}\}_{\mbk\in 2^r}$ and every $2^r$-tuple of integers
$\{j_{\mbk}\}_{\mbk\in 2^r}$,
$\bigvee_{\mbk\in 2^r}\mcal{R}_{\Psi_{\mbk}}^{j_{\mbk}}$
is forced by some $d_t$.
Then there exists a $\mbk^*\in 2^r$ such that,
let $G_s= \cup_t F^{t}_{\mbk^*,s}$ for each $s<r$, we have
$(G_0,\cdots,G_{r-1})$ 
 satisfies all requirements $\mcal{R}_\Psi^j$.
This means
 $(G_0,\cdots,G_{r-1})$ does not compute a solution of $C$. 
 Clearly $(G_0,\cdots,G_{r-1})$ is in color $\mbk^*$ of $(C_0,\cdots,C_{r-1})$.
 It is automatic
that $G_s$ is infinite for all $s<r$ since for every $n$,
there are such Turing functional
$\Psi$ that for every $m$,
$\Psi^{(G_0,\cdots,G_{r-1})}(m)\downarrow=1$ if and only if $|G_s|>n$ for all $s<r$.

As we said, it remains to prove the following.
\begin{claim}\label{limitclaim5}
There exists an extension of
$(\{\v{F}_{\mbk}\}_{\mbk\in 2^r}, \v{X})$, forcing
$\bigvee_{\mbk\in 2^r}\mcal{R}_{\mbk}$.
\end{claim}
\begin{proof}
For every $n$, consider the following set $Q_n$ of $\msf{RT}_3^1$ instance $\t{C}$
encoded by some $ (\msf{RT}_2^1)^r$ instance via the condition 
and the given Turing functionals. 
That is, for some $ (\msf{RT}_2^1)^r$ instance
$(\h{C}_0,\cdots,\h{C}_{r-1})$, $(\h{C}_0,\cdots,\h{C}_{r-1})$ together with
$(\{\v{F}_{\mbk}\}_{\mbk\in 2^r}, \v{X})$
forces $\Psi_{\mbk}^{(G_0,\cdots,G_{r-1})}\cap (n,\infty)
\subseteq \h C^{-1}(j_{\mbk})$ for all $\mbk\in 2^r$.
More specifically,  $\t C\in Q_n$ iff:
\begin{align}\nonumber
&\text{There exists a } (\msf{RT}_2^1)^r\text{ instance }(\h{C}_0,\cdots,\h{C}_{r-1})
\text{ such that}\\ \nonumber
&\ \ \text{ for every }\mbk\in 2^r,
\text{ every finite } (\h{F}_0,\cdots,\h{F}_{r-1})\in (\v{F}_{\mbk},\v{X})
\\ \nonumber
&\ \ \text{ with }(\h{F}_0\setminus F_{\mbk,0},\cdots,\h{F}_{r-1}\setminus F_{\mbk,r-1})
\text{ in color }\mbk\text{ of }(\h{C}_0,\cdots,\h{C}_{r-1}),
\\ \nonumber
& \ \ \text{ we have: }\Psi^{(\h{F}_0,\cdots,\h{F}_{r-1})}_{\mbk}\cap(n,\infty) 
\subseteq\t C^{-1}(j_{\mbk}).
\end{align}

\textbf{Case 1}.
For every $n$, $Q_n\ne 3^\omega$.

In this case, since the sequence $Q_n$ does not contain
sufficiently many $3$-colorings, by hyperimmune of $C$,
we show that $C$ will be violated by some finite extension 
of the condition.

By compactness argument and  since $Q_n$ is a $\Pi_1^{0,\v{X}}$ class uniformly in $n$
there exists an $\v{X}$-computable
 increasing array of disjoint triple of finite
 sets $\{(E_{n,0},E_{n,1},E_{n,2})\}_{n\in\omega}$, where
 each $(E_{n,0},E_{n,1},E_{n,2})$ is a partition of
$(n,l_n)\cap\omega$ for some $l_n>n$
such that:

\begin{align}\label{limiteq0}
 &\text{For every  }
 (\msf{RT}_2^1)^r
 \text{ instance }(\h{C}_0,\cdots,\h{C}_{r-1}),
\\ \nonumber
&\ \ \text{ there exists
a }\mbk\in 2^r,
\text{ a  finite  }(\h{F}_0,\cdots,\h{F}_{r-1})\in (\v{F}_{\mbk},\v{X})
\\ \nonumber
&\ \ \text{ with }
(\h{F}_0\setminus F_{\mbk,0},\cdots,
\h{F}_{r-1}\setminus F_{\mbk,r-1})
\text{ in color }\mbk \text{ of }(\h{C}_0,\cdots,\h{C}_{r-1})
\\ \nonumber
&\ \ \text{ such that }\Psi^{(\h{F}_0,\cdots,\h{F}_{r-1})}\cap (n,l_n) 
\nsubseteq E_{n,j_{\mbk}}.
\end{align}
Since $C$ is hyperimmune relative to $\v{X}$,
there exists an $n^*$ such that
$C^{-1}(j)\cap (n,l_n)=E_{n^*,j}$ for all $j\in 3$.
Taking $(\h{C}_0,\cdots,\h{C}_{r-1})$ in (\ref{limiteq0}) to be the given
$(C_0,\cdots,C_{r-1})$,
we have:
\begin{align}\nonumber
&\text{ There exists
a $\mbk^*\in 2^r$
and a  finite  $(\h{F}_0,\cdots,\h{F}_{r-1})\in (\v{F}_{\mbk^*},\v{X})$
}
\\ \nonumber
&\ \ \text{ with  $(\h{F}_0\setminus F_{\mbk^*,0},
\cdots,\h{F}_{r-1}\setminus F_{\mbk^*,r-1})$ in
color $\mbk^*$ of $(C_0,\cdots,C_{r-1})$}
\\ \nonumber
&\ \ \text{
such that $\Psi^{(\h{F}_0,\cdots,\h{F}_{r-1})}\cap
(n,l_n) \nsubseteq C^{-1}(j_{\mbk^*})$.
}
\end{align}
Thus let
$(\{\v{F}^*_{\mbk}\}_{\mbk\in 2^r},\v{X})$ be such a condition
that $\v{F}^*_{\mbk}=\v{F}_{\mbk}$ if $\mbk\ne \mbk^*$
and
$\v{F}^*_{\mbk^*}=(\h{F}_{0},\cdots,\h{F}_{r-1})$, then we are done.

\ \\

\textbf{Case 2}.
For some $n$, $Q_n=3^\omega$.

Let $Q = \{
(\t{C},\h{C}): \t{C}\in 3^\omega,
\h{C}\in (2^\omega)^r\text{ witnesses }
\t{C}\in Q_n
\}$.
Clearly $Q$ satisfies the hypothesis of
Lemma \ref{limitlem1}.
Apply (the relativized) Lemma \ref{limitlem1} to get
 $(\t{C}^0,\h{C}^0),(\t{C}^1,\h{C}^1)\in Q$
 and a $\mbk^*\in 2^r$ such that
 \begin{itemize}
 \item $(\t{C}^0,\h{C}^0)\oplus (\t{C}^1,\h{C}^1)\leq_T
 (\v{X})'$;
 \item  $\t{C}^0,\t{C}^1$ are almost disjoint, while
 \item 
 $(\h{C}^0_s)^{-1}(\mbk^*(s))\cap
(\h{C}^1_s)^{-1}(\mbk^*(s))\cap X_s =X^*_s$
is infinite for all $s<r$.
\end{itemize}
By definition of $Q$, let $\v{X}^*=(X^*_0,\cdots,
X^*_{r-1})$,
for every $(G_0,\cdots,G_{r-1})
\in (\v{F}_{\mbk^*},\v{X}^*)$, we have:
$\Psi_{\mbk^*}^{(G_0,\cdots,G_{r-1})}\cap (n,\infty)\subseteq
(\t{C}^0)^{-1}(j_{\mbk^*})\cap
 (\t{C}^1)^{-1}(j_{\mbk^*})$ is finite.
 Thus the condition
 $(\{\v{F}_{\mbk}\}_{\mbk\in 2^r},
 \v{X}^*)$ extends
 $(\{\v{F}_{\mbk}\}_{\mbk\in 2^r},
 \v{X})$ and forces
$\bigvee_{\mbk\in 2^r}\mcal{R}_{\Psi_{\mbk}}^{j_{\mbk}}$.
Thus we are done.
\end{proof}

\end{proof}

\input{sectioncrossconstraint.tex}

\section{Basis theorem for $\Pi_1^0$ class with cross 
constraint}
\label{limitsec2}

In \cite{cholak2019some}, Cholak, Dzhafarov, Hirschfeldt, Patey asked
whether $\msf{D}^2_3$ is computably reducible to 
$\msf{D}^2_2$.
This is equivalent (modulo the relativization) to ask
whether there is a $\Delta_2^0$
$3$-coloring $C\in 3^\omega$
such that for every two $\Delta_2^0$ $2$-colorings
$C_0,C_1\in 2^\omega$, there is a solution 
$(G_0,G_1)$ to $(C_0,C_1)$ such that $(G_0,G_1)$ 
does not compute a solution of $C$.
The main objective of this section is to prove the following 
improvement of Theorem \ref{limitth4} which address the above question
in affirmative.

\begin{theorem}\label{limitth4}
We have $\msf{RT}_3^1\nleq_{soc}
(\msf{RT}_2^1)^{<\omega}$. Moreover, there exists a $\Delta_2^0$ $\msf{RT}_3^1$ instance $C $
witnessing the conclusion.

\end{theorem}

Just like Theorem \ref{limitth0} is reduced to Lemma \ref{limitlem1},
Theorem \ref{limitth4} is reduced to the following
improvement of Lemma \ref{limitlem1} where 
the two members are weakened in certain ways,
i.e., preserving a sort of hyperimmune of the 
$3$-coloring $C$.
Let $Q\subseteq 3^\omega\times (2^\omega)^r$ be
a $\Pi_1^0$ class having full projection on $3^\omega$.
Let $C\in 3^\omega$ be $\Gamma$-hyperimmune (see section \ref{limitsubsec2}
for the definition). 
\begin{restatable}{lemma}{limitlem}
\label{limitlem6}
There exist $(X^0,Y^0),(X^1,Y^1)\in Q$
such that suppose $Y^i=(Y^i_0,\cdots,Y^i_{r-1})$, we have
 $X^0,X^1$ are almost disjoint
and $Y^0_s,Y^1_s$ are not almost disjoint for all $s<r$.
Moreover, $C$ is
$\Gamma$-hyperimmune relative to
$(X^0,Y^0)\oplus (X^1,Y^1) $.
\end{restatable}
The proof will be delayed to section \ref{limitsubsec4}.
Clearly, Lemma \ref{limitlem6} is a cross constraint version 
 basis theorem for $\Pi_1^0$ class since non empty $\Pi_1^0$ class
do admit members preserving the $\Gamma$-hyperimmune.
Of course, Lemma \ref{limitlem1}  can be seen as a cross constraint version
basis theorem for $\Pi_1^0$
class as well  and its normal version is simply that every non empty $\Pi_1^0$
class admit a $\Delta_2^0$ member.

The proof of Lemma \ref{limitlem6} consists of two separate
ideas:
\begin{enumerate}
\item Cross constraint version 
basis theorem for $Q$  (say cone avoidance, see 
Lemma \ref{limitlem2});

\item Defining the appropriate hyperimmune notion, namely $\Gamma$-hyperimmune,
and demonstrate how it is preserved by, say $\msf{WKL}$.

\end{enumerate}

We illustrate the idea  of cross constraint version basis theorem
in section \ref{limitsubsec1} where we prove  the cross constraint version 
 cone avoidance (Lemma \ref{limitlem2}).
In section \ref{limitsubsec2} we define the $\Gamma$-hyperimmune 
and prove in section \ref{limitsubsec3}
Lemma \ref{limitlem5} that  it is preserved by $\msf{WKL}$.
This not only serves as a demonstration of how $\Gamma$-hyperimmune is preserved
but is also needed in Lemma \ref{limitlem6}.
Combining the two ideas, we prove Lemma \ref{limitlem6} in
section \ref{limitsubsec4}.
Before we proceed to Lemma \ref{limitlem2}, we explore some 
combinatorial aspect on $Q$ in section \ref{limitsubsec0}.
We end up this part of the section by proving Theorem \ref{limitth4}
using Lemma \ref{limitlem6}.

\begin{proof}[Proof of Theorem \ref{limitth4}]
Let $C$ be a $\Delta_2^0$ $3$-coloring 
that is $\Gamma$-hyperimmune (which exists 
by Lemma \ref{limitlem7}).
Fix an $r\in\omega$ and $r$ many 
$2$-colorings $C_0,\cdots, C_{r-1}$.

We use similar \emph{condition }
$(\{\v F_{\mbk}\}_{\mbk\in 2^r}, \v X)$
as in Theorem \ref{limitth0} except that 
instead of $\v X$ being arithmetic,
we require $C$ to be $\Gamma$-hyperimmune relative 
to $\v X$.
The definition of extension, forcing
and requirement are identical as Theorem \ref{limitth0}.
It suffices to show how to extend a condition to force a 
requirement.
This part follows exactly as Claim \ref{limitclaim5}
except that:
\begin{enumerate}
\item In case 1, to see $C$ is not uniformly encoded,
we take advantage of $C$ being $\Gamma$-hyperimmune relative 
to $\v X$ (instead of hyperimmune relative to $\v X$);

\item In case 2, instead of selecting 
$(\t C^0,\h C^0), (\t C^1, \h C^1)$ so that 
$(\t C^0,\h C^0)\oplus (\t C^1, \h C^1)\leq_T (\v X)'$,
we require that $C$ is $\Gamma$-hyperimmune relative 
to $(\t C^0,\h C^0)\oplus (\t C^1, \h C^1)$,
promised by Lemma \ref{limitlem6}.

\end{enumerate}

The rest of the proof are exactly the same.

\end{proof}

\subsection{Combinatorial lemmas}
\label{limitsubsec0}

For   $\rho_0,\rho_1\in k^{<\omega}$, we say
$\rho_0,\rho_1$ are \emph{disjoint} iff 
$\rho_0^{-1}(j)\cap \rho_1^{-1}(j)=\emptyset$
for all $j\in k$.

For $n,\h n\in \omega$, a set $A\subseteq k^n$,
 an injective function $ext: A\rightarrow k^{\h n}$, we say
 $ext$ is a \emph{disjoint preserving extension over $(\rho^0,\rho^1)$}
 iff
\begin{itemize}
\item  $ext(\rho)\succeq\rho$ for all $\rho\in A$; and
\item for every $\t\rho^0,\t\rho^1\in 3^n$, if
  $(\t\rho^0,\t\rho^1)$ does not positively progress on any color $j\in 3$
compared to $(\rho^0,\rho^1)$, then
 neither does $(ext(\t\rho^0),ext(\t\rho^1))$.
\end{itemize}

\begin{lemma}\label{limitlemcomb0}
For every $k\in \omega$, $ n\leq \h n\in \omega$, 
every $\rho\in k^n,\h\rho\in k^{\h n}$,
there is a disjoint preserving extension
over $(\bot,\bot)$, namely  $ext:k^n\rightarrow k^{\h n}$
such that $ext(\rho) = \h\rho$.
\end{lemma}
\begin{proof}
The function 
$ext$ could simply be a point wise isomorphism.  
\end{proof}

Let $k\in \omega$, and for each $i\in 2$,
let $f_i:3^{<\omega}\rightarrow \mcal{P}(k)$
be two functions.
For a $\rho\in 3^{<\omega}$, a $j\in k$,
we say $(\rho,j)$ is \emph{\good}\ for $f_i$ iff:
for every $\h\rho\succeq \rho$, $j\in f_i(\h\rho)$.

\begin{lemma}\label{limitlemcomb}
 Suppose
for every $\h\rho\succeq\rho$, $f_i(\h\rho)\subseteq f_i(\rho)\ne\emptyset$
for all $i\in 2$.
Then for every $n$, every  
set $A\subseteq 3^n$,
there exist an $\h n\in \omega$,
a disjoint preserving extension over $(\bot,\bot)$
  $ext:A\rightarrow 3^{\h n}$
such that for every $\rho\in A$, every $j\in k$, every $i\in 2$,
either $(ext(\rho),j)$ is \good\ for $f_i$
or $j\notin f_i(ext(\rho))$.

\end{lemma}
\begin{proof}

Suppose $A = \{\rho_0,\cdots,\rho_{n-1}\}$.
We handle each member one at a time, by that we mean
 to extend \emph{all} members,   preserving the mutually disjoint relation,
 while the handled member $\rho_m$ is extended to some $\h\rho$ 
so that $(\h\rho,j)$ is either \good\ for $f_i$ or $(\h\rho,j)\notin f_i(\h\rho)$
for all $j\in k,i\in 2$.

At step $0$, to deal with $\rho_0$, let $\h\rho_0\succeq \rho_0$ be such that 
for every $i\in 2$, every $j\in k$,
either $(\h\rho_0,j)$ is \good\ for $f_i$ or 
$(\h\rho_0,j)\notin f_i(\h\rho_0)$
.
Such $\h\rho_0$ exists by the hypothesis on  $f_i$.
Extend each  $\rho_1,\cdots,\rho_{n-1}$
to $\h\rho_1,\cdots,\h\rho_{n-1}\in 3^{|\h\rho_0|}$  respectively
so that for every $m_0,m_1<n$,
if $\rho_{m_0},\rho_{m_1}$ are disjoint,
then $\h\rho_{m_0},\h\rho_{m_1}$ are disjoint
(such extension exists by Lemma \ref{limitlemcomb0}).

At step $1$, repeat what we do in step $0$
(with $A$ replaced by $\{\h\rho_0,\cdots,\h\rho_{n-1}\}$) to deal with 
$\h\rho_1$.

Repeat this procedure until all elements are dealt with. 
In the end, we get a set $\{\rho_0^*,\cdots,\rho_{n-1}^*\}$.
Clearly  for every $m_0,m_1<n$, if $\rho_{m_0},\rho_{m_1}$
 are disjoint, then 
 $\rho_{m_0}^*,\rho_{m_1}^*$ are disjoint
 since the mutually disjoint relation is preserved throughout each step.
 
 Now consider the function $ext(\rho_m) = \rho^*_m$
 for all $m<n$.
 It's easy to see that $ext$ satisfies the conclusion.

\end{proof}

Note that for any $n\in \omega$, the function $$f: 3^{<\omega}
\ni\rho \mapsto \{\sigma\in 2^n: [(\rho,\sigma)]\cap Q\ne\emptyset \}\in \mcal{P}(2^n)$$
satisfies the hypothesis of Lemma \ref{limitlemcomb}.

\subsection{Cross constraint version cone avoidance}
\label{limitsubsec1}
Let $Q\subseteq 3^\omega\times (2^\omega)^r$ be
a $\Pi_1^0$ class that
has full projection on $3^\omega$;
let $\ch D$ be an incomputable oracle.

\begin{lemma}\label{limitlem2}
There exists $(X^0,Y^0),(X^1,Y^1)\in Q$
such that:
suppose $Y^i=(Y^i_0,\cdots,Y^i_{r-1})$, we have
 $X^0,X^1$ are almost disjoint
and $Y^0_s,Y^1_s$ are not almost disjoint for all $s<r$.
Moreover, $(X^0,Y^0)\oplus (X^1,Y^1)\ngeq_T \ch D$.
\end{lemma}
\begin{proof}
For $\rho\in 3^{<\omega}$,
$\sigma\in 2^{<\omega}$,
for a closed set $P\subseteq  3^\omega\times 2^\omega$,
we say $(\rho,\sigma)$ is \emph{\good\ for $P$}
iff
for every $X\in [\rho]$,
 $\PX\cap [\sigma]\ne\emptyset$
 where $\PX=\{Y\in (2^\omega)^r: (X,Y)\in Q\}$.

A \emph{condition} in this lemma
is a tuple $((\rho^0,\sigma^0), (\rho^1,\sigma^1),P,D)$
such that  
\begin{itemize}
\item $D$ is an oracle that does not compute $\ch D$;
\item $P\subseteq 3^\omega\times (2^\omega)^r$ is a $\Pi_1^{0,D}$ class;
\item $(\rho^i,\sigma^i)\in 3^{<\omega}\times 2^{<\omega}$ is good 
for $P$ for all $i\in 2$.

\end{itemize}

As usual, a condition $((\rho^0,\sigma^0), (\rho^1,\sigma^1),P,D)$
is seen as a set of candidates of the desired $((X^0,Y^0),(X^1,Y^1))$
(also denoted as $((\rho^0,\sigma^0), (\rho^1,\sigma^1),P,D)$),
namely
\begin{align}\nonumber
&\big\{((\h X^0,\h Y^0),(\h X^1,\h Y^1))\in
(P\times P)\cap ([(\rho^0,\sigma^0)]\times [(\rho^1,\sigma^1)]):
\\ \nonumber
&\ \ \ (\h X^0,\h X^1)\text{ does not positively progress on any color
}j\in 3\text{ compared to }(\rho^0,\rho^1).
\big\}
\end{align}


A condition $((\h\rho^0,\h\sigma^0), (\h\rho^1,\h\sigma^1),\h P,\h D)$
\emph{extends} a condition 
$((\rho^0,\sigma^0), (\rho^1,\sigma^1),P,D)$
iff 
$$
((\h\rho^0,\h\sigma^0), (\h\rho^1,\h\sigma^1),\h P,\h D)
\subseteq 
((\rho^0,\sigma^0), (\rho^1,\sigma^1),P,D);$$
or equivalently:
\begin{itemize}
\item $(\h\rho^0,\h\rho^1)$ does not positively progress on
any color $j\in 3$ compared to $(\rho^0,\rho^1)$;

\item $(\h\rho^i,\h\sigma^i)\succeq(\rho^i,\sigma^i)$
for each $i\in 2$ and $\h P\subseteq P$.
\end{itemize}

The requirement we want to satisfy is: for each Turing functional 
$\Psi$,
$$
\mcal{R}_\Psi: \Psi^{(X^0,Y^0)\oplus (X^1,Y^1)}\ne\ch D.
$$
A condition $((\rho^0,\sigma^0), (\rho^1,\sigma^1),P,D)$ 
\emph{forces} $\mcal{R}_\Psi$
iff every member $((\h X^0,\h Y^0),(\h X^1,\h Y^1))$
in $ ((\rho^0,\sigma^0), (\rho^1,\sigma^1),P,D)$ satisfies $\mcal{R}_\Psi$.

Fix a Turing functional $\Psi$, a condition  $((\rho^0,\sigma^0), (\rho^1,\sigma^1),P,D)$,
we prove that:
\begin{claim}\label{limitclaim2}
There exists an extension of  $((\rho^0,\sigma^0), (\rho^1,\sigma^1),P,D)$
forcing $\mcal{R}_\Psi$.
\end{claim}
\begin{proof} 
For simplicity, assume $|\rho^0|=|\rho^1|, |\sigma^0|=|\sigma^1|$,
otherwise extend them to be so;
and suppose $D=\emptyset$.
Let $T_P$ be a computable tree in $3^{<\omega}\times 2^{<\omega}$ defining $P$
so that for every $\rho\in 3^{<\omega}$,
the set $\{\sigma\in 2^{<\omega}:(\rho,\sigma)\in T_P\}$
is a pruned tree.

For every oracle $\t D\in 2^\omega$,
consider the following set $\mcal{T}_{\t D}$ of trees
such that $T\in \mcal{T}_{\t D}$ iff $[T]$ is a condition forcing
$\Psi^{(X^0,Y^0)\oplus (X^1,Y^1)}=\t D$. More specifically,
$T\in \mcal{T}_{\t D}$ iff:
\begin{itemize}
\item $T\subseteq T_P$ and $(\rho^i,\sigma^i)$
is \good\ for $[T]$ for all $i\in 2$;

\item for every $(\t\rho^i,\t\sigma^i)
\in T\cap [(\rho^i,\sigma^i)]^\preceq$,
every $n\in\omega$,
if $(\t\rho^0,\t\rho^1)$ does not positively progress
on any color $j\in 3$ compared to $(\rho^0,\rho^1)$, then 
$\Psi^{(\t\rho^0,\t\sigma^0)\oplus (\t\rho^1,\t\sigma^1)}(n)\downarrow\rightarrow
\Psi^{(\t\rho^0,\t\sigma^0)\oplus (\t\rho^1,\t\sigma^1)}(n)= \t D(n)$.

\end{itemize}

The key note is that $\mcal{T}_{\t D}$ is
a $\Pi_1^{0,\t D}$ class uniformly in $\t D$.

\ \\

\textbf{Case 1.} For every $\t D\in 2^\omega$,
$\mcal{T}_{\t D}=\emptyset$.

In particular, $\mcal{T}_{\ch D}=\emptyset$.
By compactness, there exists an $N\in \omega$,
such that for 
every tree $T\subseteq 3^{\leq  N}\times 2^{\leq N}$
with $\ell(T)\subseteq 3^N\times 2^N$, if
\begin{align}\label{limiteq4}
&\text{$T\subseteq T_P$ and }
\text{ for every $\rho\in 3^{ N}$, 
}\\ \nonumber
&\ \ \ \text{
there exists a
$\sigma\in 2^N$ such that $(\rho,\sigma)\in T$},
\end{align}
then we have
that 
\begin{align}\label{limiteq6}
&\text{ there exist,  for each  }i\in 2, \text{ a }(\t\rho^i,\t\sigma^i)
\in \ell(T)\cap [(\rho^i,\sigma^i)]^\preceq,
\text{ an } m\in\omega,
\text{ such that }
\\ \nonumber
& (\t\rho^0,\t\rho^1)
\text{ does not positively progress
on any color $j\in 3$ compared to }(\rho^0,\rho^1),
\text{ and } 
\\ \nonumber
&\Psi^{(\t\rho^0,\t\sigma^0)\oplus (\t\rho^1,\t\sigma^1)}(m)\downarrow\ne \t D(m).
\end{align}

Now we illustrate how to use Lemma \ref{limitlemcomb}.
Let $k= 2^{N-|\sigma^0|}$; let $f_i:3^{<\omega}\rightarrow \mcal{P}(2^{N-|\sigma^0|})$
be such that 
for every $\rho\in 3^{<\omega}$,
\begin{align}\nonumber
f_i(\rho) = \{\sigma\in 2^{N-|\sigma^0|}: (\rho^i\rho, \sigma^i\sigma)\in T_P\}.
\end{align}
Since for each $i\in 2$, $(\rho^i,\sigma^i)$ is \good\ for $P$,
 $f_i$ satisfies the hypothesis of Lemma \ref{limitlemcomb} for all $i\in 2$.
Let $A = 3^{N-|\rho^0|}$  and let $ext:3^{N-|\rho^i|}\rightarrow 3^{<\omega} $
 be as in the conclusion of
Lemma \ref{limitlemcomb}.
 Let $g_i:3^{N-|\rho^i|}\rightarrow 2^{N-|\sigma^i|}$
 be such that $g_i(\rho)\in f_i(ext(\rho))$
for all $\rho\in 3^{N-|\rho^0|}$ and all $i\in 2$.
Which means, by definition of $ext$,
\begin{align}\label{limiteq12}
\text{ $(ext(\rho),g_i(\rho))$ is \good\ for $f_i$
for all  $\rho\in 3^{N-|\rho^0|}$ and all $i\in 2$.
}
\end{align}

Now take $T\subseteq 3^{\leq N}\times 2^{\leq N}$
to be   a tree generated by the downward closure
of the following set:
\begin{align}\nonumber
B=\{(\rho,\sigma)\in 3^N\times 2^N: 
\text{ for some }i\in 2, \text{ some }\rho', 
\rho=\rho^i\rho' \text{ and }\sigma= \sigma^i g_i(\rho')\}.
\end{align}

First we note that $T\subseteq T_P$ and satisfies (\ref{limiteq4}).
To see this, for each $(\rho,\sigma)\in B$,
suppose $\rho=\rho^i\rho'$ and $\sigma = \sigma^i g_i(\rho')$.
By definition of $g_i$, $g_i(\rho')\in f_i(ext(\rho'))\subseteq f_i(\rho')$.
Therefore, by definition of $f_i$, $(\rho,\sigma)\in T_P$.

By (\ref{limiteq6}),
let 
$(\t\rho^i,\t\sigma^i)
\in B\cap [(\rho^i,\sigma^i)]^\preceq$ be such that for some 
$n\in\omega$, $(\t\rho^0,\t\rho^1)$
does not positively progress
on any color $j\in 3$ compared to $(\rho^0,\rho^1)$
and 
$\Psi^{(\t\rho^0,\t\sigma^0)\oplus (\t\rho^1,\t\sigma^1)}(m)\downarrow\ne \ch D(m)$.

We now apply the conclusion of Lemma \ref{limitlemcomb} on $ext$ to
show that there are $\h\rho^i\succeq\t\rho^i$
such that 
\begin{itemize}
\item $(\h\rho^i,\t\sigma^i)$
is \good\ for $P$ for all $i\in 2$; and
\item  $(\h\rho^0,\h\rho^1)$
does not positively progress on any color $j\in 3$ compared to $(\rho^0,\rho^1)$.
\end{itemize}
By definition of $B$,
 suppose $\t\rho^i=\rho^i\rho'_i$ and
  $\t\sigma^i= \sigma^i g_i(\rho'_i)$.
 It is clear that $\rho'_0,\rho'_1$
 are disjoint since $(\t\rho^0,\t\rho^1)$
does not positively progress
on any color $j\in 3$ compared to $(\rho^0,\rho^1)$
(and since $|\rho^0|=|\rho^1|$).
By definition of $ext$  (where $ext$ is required to be a disjoint preserving extension
over $(\bot,\bot)$),
  \begin{itemize}
  \item $ext(\rho'_i)\succeq \rho'_i$ for all $i\in 2$;
  \item $ext(\rho'_0),ext(\rho'_1)$ are disjoint
and  
\item $(ext(\rho'_i), g_i(\rho'_i))$
is \good\ for $f_i$. 
\end{itemize}

Unfolding the definition of $f_i$, $(ext(\rho'_i),g_i(\rho'_i))$
is \good\ for $f_i$
 means for every $\rho''\in [ext(\rho'_i)]^\preceq$,
 $g_i(\rho'_i)\in f_i(\rho'')$, i.e.,
  for every $\rho''\in [ext(\rho'_i)]^\preceq$,
  $(\rho^i\rho'',\t\sigma^i)=
  (\rho^i\rho'',\sigma^i g_i(\rho'_i))\in T_P$.
  Thus, let $$\h\rho^i= \rho^i ext(\rho'_i),$$
  we have $(\h\rho^i,\t\sigma^i)$
  is \good\ for $P$.
  Obviously, $\h\rho^i\succeq \rho^i$ for all $i\in 2$.
  On the other hand, since $ext(\rho'_0),ext(\rho'_1)$
  are disjoint and $|\rho^0|=|\rho^1|$,
   therefore $(\h\rho^0,\h\rho^1)$ does not positively progress 
   on any color $j\in 3$ compared to $(\rho^0,\rho^1)$.
  
  In summary,  $((\h\rho^0,\t\sigma^0), (\h\rho^1,\t\sigma^1),P,D)$
  is the desired extension of $((\rho^0,\sigma^0),(\rho^1,\sigma^1),P,D)$
  forcing the requirement in a deterministic way.

\ \\

\textbf{Case 2.} There exists a $\t D$ such that $\mcal{T}_{\t D}\ne\emptyset$.

Note that the set of $\t D$ such that $\mcal{T}_{\t D}\ne\emptyset$
consists of a $\Pi_1^0$ class.
By cone avoidance of $\Pi_1^0$ class,
there exists a $\t D$, a $T\in \mcal{T}_{\t D}$
such that $\t D\oplus T\ngeq_T \ch D$.
Let $\h P=[T]$. 
By definition of $\mcal{T}_{\t D}$,
$(\rho^i,\sigma^i)$ is \good\ for $\h P$.
Thus,
 $((\rho^0,\sigma^0), (\rho^1,\sigma^1),\h P,\t D\oplus T)$
is a condition extending 
$((\rho^0,\sigma^0), (\rho^1,\sigma^1),P,D)$.
To see that  $((\rho^0,\sigma^0), (\rho^1,\sigma^1),\h P,\t D\oplus T)$
forces $\mcal{R}_\Psi$, we note that 
for every $((\h X^0,\h Y^0),(\h X^1,\h Y^1))\in ((\rho^0,\sigma^0), (\rho^1,\sigma^1),\h P,\t D\oplus T)$,
if $\Psi^{(\h X^0,\h Y^0)\oplus(\h X^1,\h Y^1)}$ is total,
then $\Psi^{(\h X^0,\h Y^0)\oplus(\h X^1,\h Y^1)}=\t D\ne \ch D$.
Thus we are done in this case.

\end{proof}

A condition $((\rho^0,\sigma^0), (\rho^1,\sigma^1),P,D)$
\emph{\exclude} component $s$ if it cannot  be extended
to make positive progress on  the $s^{th}$ component.
More specifically, for every condition
$((\h\rho^0,\h\sigma^0), (\h\rho^1,\h\sigma^1),\h P,\h D)$
extending $((\rho^0,\sigma^0), (\rho^1,\sigma^1))$,
 $(\h{\sigma}^0,\h{\sigma}^1)$
does not positively progress on any color
on the $s^{th}$ component compared to $(\sigma^0,\sigma^1)$.

Given a pair $(\rho,\sigma)$,
a closed set $P\subseteq 3^\omega\times (2^\omega)^r$ 
we say $(\rho,\sigma)$ \emph{lock} the $s^{th}$ component
iff: there exists a $Y\in 2^\omega$ such that 
for every 
 $(\h\rho,\h\sigma)\succeq(\rho,\sigma)$ that is 
 \good\ for $ P$, $\h\sigma_s\prec Y$
 where $\h\sigma =(\h\sigma_0,\cdots,\h\sigma_{r-1})$.

 Let $I\subseteq r$ be a maximal set such that
there exists a $D\ngeq_T \ch D$, a 
$\Pi_1^{0,D}$ class $P\subseteq Q$,
a pair $(\rho,\sigma)$ \good\ for $P$ such that 
$(\rho,\sigma)$
locks the $s^{th}$ component for all $s\in I$.
Starting with  the condition
$d_0= ((\rho,\sigma), (\rho,\sigma),P,D)$,
note that
for any extension $ d$
of $d_0$,
 $d$
does not \exclude\ any component
(see Claim \ref{limitclaim0}).

Thus combine with Claim \ref{limitclaim2},
there is a sequence of conditions
$d_0\supseteq d_1\supseteq \cdots$
such that every requirement is forced by some $d_t$;
moreover, let 
$((X^0,Y^0),(X^1,Y^1)) = \cup_t((\rho^0_t,\sigma^0_t),(\rho^1_t,\sigma^1_t))$,
we have that $Y^0_s,Y^1_s$ are not almost disjoint for all $s<r$
since we can make positive progress on each component infinitely often;
while $X^0,X^1$ are almost disjoint by definition of extension.

\end{proof}

\subsection{A notion of hyperimmune}
\label{limitsubsec2}
\def\onestep{one step variation}
\def\treecomputation{tree-computation path}
\def\tcp{tcp}
\def\approximation{approximation}
\def\overr{over}

For two finite trees
$T_0, T_1$ in $\omega^{<\omega}$, we say $T_1$ 
is a \emph{\onestep}\ of $T_0$
iff:
\begin{itemize}
\item either there is a $\xi\in \ell(T_0)$, a finite non empty set 
$B\subseteq \omega^{|\xi|+1}$
such that $T_1= T_0\cup B$;
\item or there is a non leaf $\xi\in T_0$,
a non empty set $B\subsetneq T_0\cap [\xi]^\preceq\cap \omega^{|\xi|+1}$ such that
$T_1 = (T_0\setminus [\xi]^\prec)\cup B$. 
\end{itemize}

The key fact in this definition is that the set $B$ in the or case
is non empty. This will be used in Lemma \ref{limitlem3}.

Given a poset $(W,\prec_p)$ with a root,
a \emph{\treecomputation} in $W$
is a sequence of pairs $(T_0,\phi_0),(T_1,\phi_1),\cdots$ (finite or infinite)
for every $u\in\omega$, $T_u$  is a finite tree (possibly empty) in $\omega^{<\omega}$;
$\phi_u: T_u\cup \ell(T_u)\rightarrow W$ is a function such that,
\begin{itemize}
\item $T_0= \emptyset$;
\item for  every $\h\xi,\xi\in T_u$
with $\h\xi\succ\xi$, $\phi_u(\h\xi)\succ_p \phi_u(\xi)$;

\item   $T_{u+1}$ is a  \onestep\
of $T_u$; moreover,

\item $\phi_{u+1}\uhr T_{u+1}\cap T_u = \phi_u\uhr T_{u+1}\cap T_u$.

\item $\phi_u(\bot) = $ the root of $W$.

\end{itemize}
In our application, the function $\phi_u$
will be given by a Turing functional we  diagonal against;
a  tree in the sequence represent a tree of initial segments.
The either case of \onestep\ represents the case that for some 
initial segment, which is the leaf of the tree, sufficiently many
extension makes the Turing functional's computation progress;
the or case  represents that it is found that some initial 
segment cannot be extended to be the next condition.

A simple observation is 
\begin{lemma}\label{limitlem3}
If $(W,\prec_p)$ is well founded,
 then there is no infinite \treecomputation\ in $W$.
\end{lemma}
\begin{proof}
The key fact here is that in the or case of \onestep,
the set $B$ is non empty.
Therefore, suppose $(T_0,\phi_0),\cdots$
is an infinite \treecomputation\ in $W$.
Then there is an $X\in \omega^\omega$ and a
sequence of integers $u_0<u_1<\cdots$
 such that $X\uhr n\in T_{v}$
for all $n,v\geq u_n$.
Since $\phi_{u+1}$ inherits $\phi_u$ by definition of 
\treecomputation, therefore
$\phi_{u_{n+1}}(X\uhr n+1)\succ_p \phi_{u_n}(X\uhr n)$.
Thus,
a contradiction to the well foundness of $W$.
\end{proof}

\begin{definition}[The $\Gamma_m$ space]
We inductively define the following set $\Gamma_m$.
\begin{enumerate}
\item
Let $\Gamma_0$ be the set of partial functions 
from $\omega$ to $3$ with finite domain.
Although $\Gamma_0$ admit a natural partial order,
we here define $\prec_{p_0}$ on $\Gamma_0$
so that every two elements of $\Gamma_0$
 are incomparable provided they have non empty domain;
 and the root of $\Gamma_0$ is the partial function
 with empty domain.

\item Suppose we have inductively defined $\Gamma_0,\cdots,\Gamma_{m-1}$
where each is a poset with a root.
Define 
\begin{align}\nonumber
\text{ $\Gamma_m$ as the set of
 finite \treecomputation\ in $\Gamma_{m-1}$.}
 \end{align}
The root of $\Gamma_m$ is clearly
the singleton $(\emptyset,\phi)$
(by definition of \treecomputation,
$\phi$ is the function with domain $\{\bot\}=\ell(\emptyset)$
and $\phi(\bot) = $ the root of $\Gamma_{m-1}$). 
 \item Since each element of $\Gamma_m$ is a sequence, it makes
 sense to say one \treecomputation\ in $\Gamma_{m-1}$ 
 is an initial segment of the other.
 This give rise to a natural partial order $\prec_{p_m}$ on $\Gamma_m$
 where $\tcp_0\prec_{p_m}\tcp_1$ means 
 $\tcp_0$ is an initial segment of $\tcp_1$.

\item For a $c\in \Gamma_0$, we say $c$ is \emph{\overr}
$n$ if $dom(c)\subseteq (n,\infty)$.

\item Suppose we have defined \overr\ $n$  for elements in $\Gamma_{m-1}$,
for a 

$\tcp= ((T_0,\phi_0),\cdots,(T_{u-1},\phi_{u-1}))\in \Gamma_m$, we say $\tcp$ is \overr\ $n$
if for every $v<u$, every $\xi\in T_v\cup\ell(T_v)$,
$\phi_v(\xi)$ is \overr\ $n$.

\end{enumerate}
\end{definition}

According to Lemma \ref{limitlem3}, by induction,  it is direct to see
\begin{lemma}\label{limitlem4}
For every $m$, $\Gamma_m$
is well founded.
\end{lemma}

\begin{definition}[Diagonal against]
For a partial function $C$ from $\omega$ to $3$,
\begin{enumerate}
\item 
We say $C$ \emph{diagonal against} a $c\in \Gamma_0$ iff
$C\uhr dom(c) = c$.

\item Suppose we have defined what it means
 for $C$ diagonal against $\tcp$ when  $\tcp\in \Gamma_{m-1}$.
 For a 
 $\tcp= ((T_0,\phi_0),\cdots,(T_{u-1},\phi_{u-1}))\in \Gamma_m$,
 we say $C$ \emph{diagonal against}  $\tcp$
 iff  
 there exists a $\xi\in \ell(T_{u-1})$,
 such that $C$ diagonal against $\phi_{u-1}(\xi)$.

\end{enumerate}

\end{definition}
Intuitively, $C$ diagonal against $\tcp$ means
$C$ agree with some partial function $\omega\rightarrow 3$
involved in some top node of $T_{s-1}$.
It can be shown by induction that if 
$\tcp\in \Gamma_m$ is over $n$, then there is a 
partial function $C$ from $\omega$ to $3$ 
with a finite domain $dom(C)\subseteq (n,\infty)$
such that $C$ diagonal against $\tcp$.

\begin{definition}[$\Gamma$-approximation]

For each $m\in\omega$, a $\Gamma_m$-\emph{approximation}
is a function $f:\omega\times \omega\rightarrow \Gamma_m$ such that 
for every $n,s\in\omega$,  $f(n ,s)$ is \overr\ $n$,
$f(n,s+1)\succeq_{p_m} f(n,s)$ and $f(n,0) = (\emptyset,\phi)$.
A \emph{$\Gamma$-approximation} is a $\Gamma_m$-approximation for some $m$.

\end{definition}

Note that by  Lemma \ref{limitlem4},
for every $\Gamma$-approximation $f$, every 
$n$, there is an $s$ so that 
the computation of $f(n,\cdot)$ converges
to $f(n,s)$, i.e., $f(n,t)=f(n,s)$
for all $t\geq s$; denote  this $f(n,s)$ as $f(n)$.

\begin{definition}[$\Gamma$-hyperimmune]
A $3$-coloring $C$ is \emph{$\Gamma$-hyperimmune} relative 
to a Turing degree $D$
 iff for every $D$-computable $\Gamma$-approximation $f$,
 there exists an $n\in\omega$ such that 
  $C$ diagonal against $f(n)$.
 When $D$ is computable, we simply say $C$
 is $\Gamma$-hyperimmune.

\end{definition}

Obviously, $\Gamma$-hyperimmune generalizes the idea of hyperimmune
and implies hyperimmune. 
Due to Lemma \ref{limitlem4},
\begin{lemma}\label{limitlem7}
There exists a $\Delta_2^0$
$3$-coloring $C\in 3^\omega$ that is 
$\Gamma$-hyperimmune.

\end{lemma}
\begin{proof}
Note that 
we can computably enumerate all computable $\Gamma$-approximations
$f_0,f_1,\cdots$.
Suppose we have defined $C$ on $[0,n]\cap \omega$.
To diagonal against a $\Gamma_m$-approximation $f$,
simply $\emptyset'$-compute $f(\h n)$ for some $\h n\geq n$.
That is, 
thanks to Lemma \ref{limitlem4},  we can 
$\emptyset'$-compute the $s$ such that $f(\h n,s)=f(\h n,t) $
for all $t\geq s$.

\end{proof}

\subsection{Preservation of the $\Gamma$-hyperimmune
for $\Pi_1^0$ class}
\label{limitsubsec3}
Fix a $\Gamma$-hyperimmune $3$-coloring $C$,
a non empty $\Pi_1^0$ class $Q\subseteq 2^\omega$.
\begin{lemma}\label{limitlem5}
There exists an $X\in Q$
such that $C$ is $\Gamma$-hyperimmune relative to $X$. 

\end{lemma}
\begin{proof}
The \emph{condition } we use 
in this lemma is a pair 
$(\rho,P)$ where 
$\rho\in 2^{<\omega}$,
$P$ is a $\Pi_1^{0,D}$ class 
for some Turing degree $D$ 
so that $C$ is $\Gamma$-hyperimmune relative to $D$;
and $P\subseteq [\rho]$.

Each condition $(\rho,P)$ is seen as a collection
of candidates of the $X$ we construct, namely
$P$. 

A condition $(\h\rho,\h P)$ \emph{extends } a 
condition $(\rho,P)$ iff $\h P\subseteq P$.
A condition $(\rho,P)$ \emph{forces} a requirement $\mcal{R}$
iff for every $Y\in P$, $Y$ satisfies  $\mcal{R}$.

The requirement we deal with is, for each Turing functional
$\Psi$,
$$
\mcal{R}_{\Psi}:\text{ for some }n,
C\text{ diagonal against }\Psi^X(n).
$$
Here we adopt the convention that 
for each Turing functional $\Psi$, there exists a $m$,
such that 
for every oracle $Y$, every $n,s$,
$\Psi^Y(n,s)\in \Gamma_m$.

Starting with the condition $(\bot, Q)$,
we will  build a sequence of conditions
so that each requirement is satisfied, then 
  the common element of these conditions satisfies
  all requirements.
  Therefore, it suffices to show how to extend
  a given condition to force a given requirement.
  Fix a Turing functional $\Psi$,
  a condition $(\ch\rho,P)$.
  
  \begin{claim}\label{limitclaim4}
  There exists an extension of $(\ch\rho,P)$ forcing $\mcal{R}_{\Psi}$.
  \end{claim}
\begin{proof}
We adopt the convention that for every $\rho$,
every $n$, $\Psi^\rho(n,s)$ is defined iff 
$s\leq |\rho|$; and we write $\Psi^\rho(n)$ for $\Psi^\rho(n)$.
Assume that for every oracle  $Y$, every $n,s\in\omega$,
$\Psi^Y(n,s)\in \Gamma_{m-1}$;
and $\Psi^Y(n,0)$ is the root of $\Gamma_{m-1}$.
For notation simplicity, suppose $P$ is a $\Pi_1^{0,D}$ class
where $D=\emptyset$.
Let $T_P$ be a pruned, co-c.e. tree so that $P=[T_P]$.

We observe the behavior of $\Psi$ on oracles $\rho\in T_P$,
and define an $\Gamma_m$-approximation $f$ as following.
Fix an $n$.

\begin{definition}[Computing $f(n,\cdot)$]
We define $f(n,s)$ by induction on $s$.
\begin{enumerate}
\item By convention,
 $f(n,0) $ is the sequence consisting of
 $(\emptyset,\phi)$ alone.

\item Suppose we have defined 
$f(n,s) = \tcp=((T_0,\phi_0),\cdots,(T_{u-1},\phi_{u-1}))\in\Gamma_m$
and together with an injective function
$\psi:T_{u-1}\cup\ell(T_{u-1})\rightarrow 2^{<\omega}\cup\{\bot\}$
with $\psi(\bot)=\bot$ so that
\begin{align}\nonumber
&\ P \subseteq [\psi(\ell(T_{u-1}))];\\ \nonumber
&\text{ for every }\xi\in T_{u-1}, \phi_{u-1}(\xi) = \Psi^{\psi(\xi)}(n).
\\ \nonumber
&\text{ for every }\h\xi,\xi\in T_{u-1}\cup\ell(T_{u-1}),
\text{ if }\h\xi\succ\xi, \text{ then }\psi(\h\xi)\succ \psi(\xi).
\end{align}

\item Define $f(n,s+1)$ depending on which of the following cases
occur:
\begin{align}\label{limiteq8}
&\text{for some }\xi\in \ell(T_{u-1}),
\text{ for every }\rho\in [\psi(\xi)]^\preceq\cap 
T_P[s+1]\cap 2^{s+1},
\\ \nonumber
&\ \ \Psi^\rho(n)\succ_{p_{m-1}}\phi_{u-1}(\xi);
\\  \label{limiteq9}
&\text{for some }\xi\in T_{u-1}\cup\ell(T_{u-1}),
[\psi(\xi)]^\preceq\cap T_P[s+1]=\emptyset;
\\ \label{limiteq10}
&\text{otherwise}.
\end{align}

\item 
\begin{itemize}
\item In case of (\ref{limiteq8}),
suppose $[\psi(\xi)]^\preceq\cap
T_P[s+1]\cap 2^{s+1}= \{\rho_0,\cdots,\rho_{w-1}\}$.
We add a set $B\subseteq \omega^{|\xi|+1}\cap [\xi]^\preceq$
of size $w$ to $T_{u-1}$ getting $T_u$;
then define $\psi$ on $T_u$ by inheriting
$\psi$  
 and 1-1 mapping $B$  to $\{\tau_0,\cdots,\tau_{w-1}\}$;
finally, define $\phi_u$ by inheriting 
$\phi_{u-1}$ and $\phi_u(\h\xi) = \Psi^{\psi(\h\xi)}(n)$
for all $\h\xi\in B$.
Let $f(n,s+1) = \tcp^\smallfrown (T_u,\phi_u)$.
Clearly $f(n,s+1)$ is a \treecomputation\ in $\Gamma_{m-1}$
since $T_u$ is a \onestep\ of $T_{u-1}$
(fulfilling the either case). 
The inductive hypothesis on $\psi$ is easily verified.

\item  In case of (\ref{limiteq9}), suppose 
$\xi$ is the minimal string witness (\ref{limiteq9}),
i.e., non proper initial segment of $\xi$ is a witness.
Let $\h\xi$ be the immediate predecessor 
of $\xi$ and suppose 
$T_{u-1}\cap [\h\xi]^\preceq \cap \omega^{|\xi|} = 
B$.
Define $T_u$ as $(T_{u-1}\setminus [\h\xi]^\prec)\cup B$;
let $\h\psi,\phi_u$ inherit $\psi,\phi_{u-1}$ on $T_u$
respectively. Let $f(n,s+1) = \tcp^\smallfrown (T_u,\phi_u)$.
By minimality of $\xi$, $P\subseteq [\psi(\ell(T_u))]$.
The other part of inductive hypothesis is easily verified (with respect to 
$\h\psi,\phi_u$).

\item In case of (\ref{limiteq10}),
let $f(n,s+1)=f(n,s)$.

\end{itemize}

\end{enumerate}

\end{definition}

As we argued in  item (4)
of the definition of $f(n,\cdot)$,
we have $f$ is a $\Gamma_m$-approximation.
By compactness,
$f$ is computable. Since, $C$ is 
$\Gamma$-hyperimmune, there are
$n^*,s^*$ such that 
\begin{itemize}
\item $C$ diagonal against $f(n^*,s^*)$
and 
\item $f(n^*,t) =f(n^*,s^*)$ for all $t\geq s^*$.
 
\end{itemize}

Suppose $f(n^*,s^*)=((T_0,\phi_0),\cdots,(T_{u-1},\phi_{u-1}))$.
By definition of diagonal against,
there is a $\xi\in \ell(T_{u-1})$
such that $C$ diagonal against $\phi_{u-1}(\xi)$.
Consider $\rho^*= \psi(\xi)$
and the $\Pi_1^0$ class $P^*\subseteq P\cap [\rho]$ on which the computation
of $\Psi(n^*,\cdot)$ converges at $\Psi^{\rho^*}(n^*)$,
that is,
 $Y\in P^*$
iff: 
\begin{align}\nonumber
\text{ $Y\in P\cap [ \rho^*]$ and 
$\Psi^Y(n^*,s) = \Psi^{\rho^*}(n^* )$
for all $s\geq |\rho^*|$.
}\end{align}
Note that by our definition of $f$,
since $f(n^*,t) = f(n^*,s^*)$ for all $t\geq s^*$,
this means  
(\ref{limiteq8}) no longer  occurs after step $s^*$.
Therefore, $P^*\ne\emptyset$.

But $\phi_{u-1}(\xi) = \Psi^{\rho^*}(n^* )$,
therefore
we have $C$ diagonal against $\Psi^Y(n^*)$ for all $Y\in P^*$.
Thus $(\rho^*,P^*)$ is the desired extension of $(\ch \rho, P)$
forcing $\mcal{R}_\Psi$.

\end{proof}

\end{proof}

\subsection{Cross constraint version of the $\Gamma$-hyperimmune
preservation}
\label{limitsubsec4}

\def\Tcross{T^{\times}}
\def\subsufficient{sub-sufficient}

Fix a $\Gamma$-hyperimmune $3$-coloring $C$,
a $\Pi_1^0$ class $Q\subseteq 3^\omega\times (2^\omega)^r$
having full projection on $3^\omega$.
\limitlem*
\begin{proof}
The definition of condition, extension,
forcing are the same as Lemma \ref{limitlem2}.
The requirement is as in Lemma \ref{limitlem5}.
Following the proof of Lemma \ref{limitlem2}, it suffices 
to show that any condition can be extended to force a given requirement.
Fix a condition $((\ch\rho^0,\ch\sigma^0),(\ch\rho^1,\ch\sigma^1),P,D)$
and a Turing functional $\Psi$.
\begin{claim}
There is an extension of  
$((\ch\rho^0,\ch\sigma^0),(\ch\rho^1,\ch\sigma^1),P,D)$
forcing $\mcal{R}_\Psi$.

\end{claim}
\begin{proof}
For simplicity, assume  $|\ch\rho^0|=|\ch\rho^1|, |\ch\sigma^0|=|\ch\sigma^1|$,
otherwise extend them to be so.
In this proof, whenever we write $(\rho,\sigma)$ it automatically implies
$|\rho|\geq |\sigma|$.
We adopt the similar  convention as in Claim \ref{limitclaim4},
i.e., for every  
  $(\rho,\sigma)\in 3^{<\omega}\times 2^{<\omega}$,
every $n$, $\Psi^{(\rho,\sigma)}(n,s)$ is defined iff
$s\leq |\sigma|$;
we write $\Psi^{(\rho,\sigma)}(n)$ for
$\Psi^{(\rho,\sigma)}(n,|\sigma|)$;
for some $m$,
 for every oracle  $Y$, every $n,s\in\omega$,
$\Psi^Y(n,s)\in \Gamma_{m-1}$;
 $D=\emptyset$;
$T_P$ is a pruned, co-c.e. tree so that $P=[T_P]$.

Let $\Tcross$ denote the set of  pairs
$((\rho^0,\sigma^0),(\rho^1,\sigma^1))\in (T_P\cap [(\ch\rho^0,\ch\sigma^0)]^\preceq)\times
(T_P\cap [(\ch \rho^1,\ch \sigma^1)]^\preceq)$
such that $(\rho^0,\rho^1)$ does not positively
progress on any color $j\in 3$ compared to 
$(\ch \rho^0,\ch\rho^1)$.

The proof generally follows that of Lemma \ref{limitlem5}
except that we need  to incorporate the combinatorics
of $P$ as in Lemma \ref{limitlem4} in the following way.
In the computation of $f(n)$:
\begin{itemize}
\item The range of function $\psi$ is $\Tcross$
instead of $2^{<\omega}$.
\item Case (\ref{limiteq10}) can be interpreted as for some 
$\xi\in T_{u-1}\cup \ell(T_{u-1})$, it is found that 
the initial segment $\psi(\xi)$ can not be extended to be 
the next condition.
 In this Lemma, this becomes for some $\xi\in T_{u-1}\cup\ell(T_{u-1})$,
for the initial segment 
 $ ((\rho^0,\sigma^0),(\rho^1,\sigma^1))$ corresponding to $\xi$,
 $(\rho^i,\sigma^i)$ is not \good\ for $P$
for some $i\in 2$.
We will delete such $\xi$ just like in Lemma  \ref{limitlem5}.

\item Case (\ref{limiteq9}) can be interpreted as 
there are ``sufficiently" many extension $\h\rho$ of
$\psi(\xi)$ that make the Turing functional 
$\Psi^{\h\rho}(n)$ progress (compared to $\Psi^{\psi(\xi)}(n)$).
By ``sufficient" many, it  means: provably, one of them can be extended
to be the next condition's initial segment.

We have a similar case in this Lemma except that the definition 
of ``sufficient" becomes: provably, one of them can be extended
to some \good\ pairs as in Lemma \ref{limitlem2} case 1.
i.e., by our definition \ref{limitdefsuf} of sufficient,
Lemma \ref{limitlem2} case 1 can be rephrased as 
there are sufficiently many  extensions $((\t\rho^0,\t\sigma^0),(\t\rho^1,\t\sigma^1))$
of $((\ch\rho^0,\ch\sigma^0),(\ch\rho^1,\ch\sigma^1))$ such that
$\Psi^{((\t\rho^0,\t\sigma^0),(\t\rho^1,\t\sigma^1))}(n)\ne \t D(n)$ for some $n$

\item Since we here have a very different definition of 
``sufficient" (see definition \ref{limitdefsuf}) than that of Lemma \ref{limitlem5},
 the key point is that in 
case (\ref{limiteq10}), how the ``sufficient" of a set $A$ is
preserved while deleting certain element from $A$.
What we actually do is to: firstly extends the  $3^{<\omega}$-component of 
the members in $A$ in a disjoint preserving manner
so that the non \good\ member's component is extended to the witness of 
its non \good; then delete the element 
not in $T_P\times T_P$. We prove that this action
produces a \subsufficient\ set (see Claim \ref{limitclaim3}).

\end{itemize}

\begin{definition}[Sufficient]\label{limitdefsuf}
Given a set $A\subseteq (3^{<\omega}\times 2^{<\omega})\times(3^{<\omega}\times 2^{<\omega})$,
a $((\rho^0,\sigma^0),(\rho^1,\sigma^1))\in (3^{<\omega}\times 2^{<\omega})\times
(3^{<\omega}\times 2^{<\omega})$ with $|\rho^0|=|\rho^1|,|\sigma^0|=|\sigma^1|$,
we say $A$ is a \emph{sufficient} over $((\rho^0,\sigma^0),(\rho^1,\sigma^1))$
iff there does not exist a tree $T$ such that 
\begin{itemize}
\item $T\subseteq T_P$ and $(\rho^i,\sigma^i)$
is \good\ for $[T]$ for all $i\in 2$; and

\item for every $(\t\rho^i,\t\sigma^i)
\in T\cap [(\rho^i,\sigma^i)]^\preceq$,
if $(\t\rho^0,\t\rho^1)$ does not positively progress on any color $j\in 3$ compared to $(\rho^0,\rho^1)$,
 then $((\t\rho^0,\t\sigma^0),(\t\rho^1,\t\sigma^1))\notin A$.

\end{itemize}

We say $A$ is \emph{\subsufficient} over  $((\rho^0,\sigma^0),(\rho^1,\sigma^1))$ iff
there exists a $((\t\rho^0,\t\sigma^0),(\t\rho^1,\t\sigma^1))\in A$
with $(\t\rho^i,\t\sigma^i)\succeq (\rho^i,\sigma^i)$ for all $i\in 2$, a
$\h\rho^i\succeq\t\rho^i$ for each $i\in 2$ such that
$(\h\rho^0,\h\rho^1)$ does not positively progress on any color $j\in 3$
compared to $(\rho^0,\rho^1)$ and
 $(\h\rho^i,\t\sigma^i)$ is \good\ for $ P$ for all $i\in 2$.
\end{definition}

Note that in terms of sufficient,
case 1 of Lemma \ref{limitlem2} can be rephrased as 
for every $\t D$, the set $A$ consisting
  of $((\t\rho^0,\t\sigma^0),(\t\rho^1,\t\sigma^1))$
   such that  
$\Psi^{((\t\rho^0,\t\sigma^0),(\t\rho^1,\t\sigma^1))}(n)\ne \t D(n)$ for some $n$
is sufficient over the initial segment of the given condition.

 
 For every partial function $e$ on $3^{<\omega}$,
 $e$ give rise to a function on $(3^{<\omega}\times 2^{<\omega})
 \times (3^{<\omega}\times 2^{<\omega})$, also denoted as $e$,
 such that $e((\rho^0,\sigma^0),(\rho^1,\sigma^1))
  = ((e(\rho^0),\sigma^0), (e(\rho^1),\sigma^1))$.

For   a set $A\subseteq (3^{<\omega}\times 2^{<\omega})
\times (3^{<\omega}\times 2^{<\omega})$,
a $\rho\in 3^{<\omega} $,
we say $\rho$ is \emph{involved} in $A$
iff there exists a $((\rho^0,\sigma^0),(\rho^1,\sigma^1))\in A$
such that $\rho=\rho^i$ for some $i\in 2$.
Let $Involve(A)$ denote the set of strings $\rho\in 3^{<\omega}$ involved in $A$;
similarly for $(\rho,\sigma)$ is involved in $A$.

\begin{claim}\label{limitclaim3}
Suppose $(\rho^i,\sigma^i)$ is \good\ for $P$.
Suppose $A\subseteq   (3^{N_0}\times 2^{N_1})\times (3^{N_0}\times 2^{N_1})$ 
is sufficient over $((\rho^0,\sigma^0),(\rho^1,\sigma^1))$
where $|\rho^0|=|\rho^1|\wedge |\sigma^0|=|\sigma^1|$.
Suppose $ext:Involve(A)\rightarrow 3^{\t N}$ is a disjoint preserving extension over 
$(\rho^0,\rho^1)$.
Then
  the set  $ext(A)\cap (T_P\times T_P)$
 is \subsufficient\ over
$((\rho^0,\sigma^0),(\rho^1,\sigma^1))$.
\end{claim}
\begin{proof}

By Lemma \ref{limitlemcomb} there is a 
disjoint preserving extension over $(\rho^0,\rho^1)$,
namely $\t{ext}:ext(Involve(A))\rightarrow 3^{\h N}$ for some $\h N$
such that
\begin{align}\label{limiteq18}
&\text{ for every }\rho\in ext(Involve(A)),
\text{ every }\sigma\in 2^{N_1},
\\ \nonumber
&\text{ either }(\t{ext}(\rho),\sigma)\notin T_P
\text{ or }(\t{ext}(\rho),\sigma)\text{ is \good\ for }P.
\end{align}
Consider $$\h{ext}= \t{ext}\circ ext.$$
Clearly $\h{ext}$ is a disjoint preserving extension over 
$(\rho^0,\rho^1)$ since it is the composition of two 
disjoint preserving extension over $(\rho^0,\rho^1)$.
We show that for some $\t\tau\in ext(A)\cap (T_P\cap T_P)$,
$\h\tau = \t{ext}(\t\tau)$ and $\t\tau$ witness
the sub-sufficient over $((\rho^0,\sigma^0),(\rho^1,\sigma^1))$
 of $ext(A)\cap  (T_P\cap T_P)$.

Consider the set $B$
 such that 
 $(\rho,\sigma)\in B$
iff: 
\begin{align}\nonumber
\text{$\rho\in 3^{N_0}$; and  when $\rho\in Involve(A)$,
 $(\h{ext}(\rho),\sigma)$
is \good\ for $P$.}
\end{align}
Since  for every $i\in 2$,
 $(\rho^i,\sigma^i)$ is \good\ for $P$, we have
 for every $\rho\in 3^{\h N}\cap [\rho^i]^\preceq$,
there exists a $\sigma\in 2^{N_1}\cap [\sigma^i]^\preceq$ such that $(\rho,\sigma)\in T_P$.
In particular, for every $\rho\in Involve(A)$,
there exists a $\sigma\in  2^{N_1}\cap [\sigma^i]^\preceq$
such that $(\h{ext}(\rho),\sigma)\in T_P$, which means by (\ref{limiteq18}),
$(\h{ext}(\rho),\sigma)$ is \good\ for $P$.
Therefore, 
\begin{align}\nonumber
&\text{ for every $i\in 2$,  every $X\in[\rho^i] $,
there exists a $\sigma\in 2^{N_1}\cap [\sigma^i]^\preceq$}\\ \nonumber
&\text{
such that $(\rho,\sigma)\in B$ and }
\PX\cap [\sigma]\ne\emptyset.
\end{align}
Therefore,  we can enlarge $B$ to a tree $T\subseteq T_P$
such that 
\begin{itemize}
\item $(\rho^i,\sigma^i)$ is \good\ for $[T]$ for all $i\in 2$; and
\item $T\cap 3^{N_0}\times 2^{N_1} = B$.
\end{itemize}
Since $A$ is sufficient over  $((\rho^0,\sigma^0),(\rho^1,\sigma^1))$ and 
$A\subseteq   (3^{N_0}\times 2^{N_1})\times (3^{N_0}\times 2^{N_1})$, there exists  a
$((\t\rho^0,\t\sigma^0),(\t\rho^1,\t\sigma^1))\in A$
such that
\begin{itemize}
\item  $(\t\rho^i,\t\sigma^i)\in B$ and $(\t\rho^i,\t\sigma^i)\succeq(\rho^i,\sigma^i)$
for all $i\in 2$;
\item $(\t\rho^0,\t\rho^1)$ does not positively progress on any color $j\in 3$ compared to $(\rho^0,\rho^1)$.
\end{itemize}
By definition of $B$ and since 
$\t\rho^0,\t\rho^1\in Involve(A)$,
we have $(\h{ext}(\t\rho^i),\t\sigma^i)$
is \good\ for $P$; by definition of $\h{ext}$,
$(\h{ext}(\t\rho^0),\h{ext}(\t\rho^1))$
does not positively progress on any color $j\in 3$ compared to $(\rho^0,\rho^1)$
and $\h{ext}(\t\rho^i)\succeq ext(\t\rho^i)$
for all $i\in 2$.
Also note that for every $i\in 2$,
 $(ext(\t\rho^i),\t\sigma^i)\in T_P$ since 
$(\h{ext}(\t\rho^i),\t\sigma^i)\in T_P$ while $\h{ext}(\t\rho^i)\succeq ext(\t\rho^i)$
for all $i\in 2$.
Thus, $((ext(\t\rho^0),\t\sigma^0),(ext(\t\rho^1),\t\sigma^1))
\in ext(A)\cap (T_P\times T_P)$ witnesses
the \subsufficient\ of $ext(A)\cap (T_P\times T_P)$ and we are done.

\end{proof}

We observe the behavior of $\Psi$ on oracles
in the condition
$((\ch\rho^0,\ch\sigma^0),(\ch\rho^1,\ch\sigma^1),P,D)$ 
and define an $\Gamma_m$-approximation $f$ as following.
Fix an $n$.

\begin{definition}[Computing $f(n,\cdot)$]
We define $f(n,s)$ by induction on $s$.
\begin{enumerate}
\item By convention,
 $f(n,0) $ is the singleton
 $(\emptyset,\phi)$.

\item Suppose we have defined
$f(n,s) = \tcp=((T_0,\phi_0),\cdots,(T_{u-1},\phi_{u-1}))$,
an injective function
$\psi:T_{u-1}\cup\ell(T_{u-1})\rightarrow \Tcross$.
Moreover, 
 for each non leaf $\xi\in T_{u-1}$, 
 suppose we have assigned:
 a set $A_\xi\subseteq \Tcross\cap (3^{N_0}\times 2^{N_1})\times (3^{N_0}\times 2^{N_1})$
for some $N_0,N_1\in\omega$; 
a disjoint preserving extension over $(\ch\rho^0,\ch\rho^1)$
 $ext_\xi: 
Involve(A_\xi)\rightarrow 3^{\h N}$ for some $\h N$,
such that let $\t\xi$ be the immediate predecessor of $\xi$
\footnote{In case $\xi=\bot$, $ ext_{\t\xi}(\psi(\xi)) = ((\ch\rho^0,\ch\sigma^0),(\ch\rho^1,\ch\sigma^1))$},
\begin{itemize}
\item  $A_\xi
  \text{ is sufficient over }
ext_{\t\xi}(\psi(\xi));$
\item 
$ \psi$\text{ restricted on }$T_{u-1}\cap \omega^{|\xi|+1}$
\text{ is a one-one function with the range } $
\{\tau\in A_\xi: ext_\xi(\tau)\in (T_P[s]\times T_P[s])\}$;
\text{ and }$\psi(\bot) = ((\ch\rho^0,\ch\sigma^0),(\ch\rho^1,\ch\sigma^1))$;
 
\item for every $\h\xi,\xi\in T_{u-1}\cup\ell(T_{u-1}),
\text{ if }\h\xi\succ\xi, \text{ then }
\psi(\h\xi)\succ ext_{\t\xi}(\psi(\xi))$; 

\item  for every 
$\bot\ne\xi\in T_{u-1}\cup\ell(T_{u-1}), \phi_{u-1}(\xi) = \Psi^{ext_{\t\xi}(\psi(\xi))}(n)$.

\end{itemize}

\item Define $f(n,s+1)$ depending on which of the following cases
occur: for some non root $\xi\in T_{u-1}\cup \ell(T_{u-1})$,
suppose $\t\xi$ is the immediate predecessor of $\xi$
and $\psi(\xi) = ((\rho^0,\sigma^0),(\rho^1,\sigma^1))$,
\begin{align}\label{limiteq14}
\hspace{1cm}&\xi\in \ell(T_{u-1})
\text{ and it is found at this time that the set  }
\\ \nonumber
&\ \ \big\{\tau\in \Tcross
\cap [ext_{\t\xi}(\psi(\xi))]^\prec:
\Psi^\tau(n)\succ_{p_{m-1}}\phi_{u-1}(\xi)
\big\}\text{ is sufficient over }
ext_{\t\xi}(\psi(\xi));
\\  \label{limiteq15}
&
\text{for some }i\in 2, (ext_{\t\xi}(\rho^i),\sigma^i)\text{ is not \good\ for 
}[T_P[s+1]];
\\ \label{limiteq16}
&\text{otherwise}.
\end{align}

\item
\begin{itemize}
\item In case of (\ref{limiteq14}),
by compactness, let $A_\xi\subseteq ((3^{N_0}\times 2^{N_1})
\times(3^{N_0}\times 2^{N_1}))\cap \Tcross$
be a witness for the sufficient over $ext_{\t\xi}(\psi(\xi))$ of that set.
Let $ext_\xi: Involve(A_\xi)\rightarrow 3^{N_0}$
be an identity function.
We add a set $B\subseteq \omega^{|\xi|+1}\cap [\xi]^\preceq$
of size $|A_\xi|$ to $T_{u-1}$ getting $T_u$;
then define $\psi$ on $T_u$ by inheriting
$\psi$
 and 1-1 mapping $B$  to $A_\xi$;
finally, define $\phi_u$ by inheriting
$\phi_{u-1}$ and $\phi_u(\h\xi) = \Psi^{ext_\xi(\psi(\h\xi))}(n )$
for all $\h\xi\in B$.
Let $f(n,s+1) = \tcp^\smallfrown (T_u,\phi_u)$.
Clearly $f(n,s+1)$ is a \treecomputation\ in $\Gamma_{m-1}$
since $T_u$ is a \onestep\ of $T_{u-1}$
(fulfilling the either case).
The inductive hypothesis 
 on $\psi$ is easily verified.

\item  In case of (\ref{limiteq9}), suppose
$\xi$ is the minimal string witness (\ref{limiteq9}),
i.e., no proper initial segment of $\xi$ is a witness.
Without loss of generality, suppose for some 
$\h\rho\succeq ext_{\t\xi}(\rho^0)$,
$(\h\rho,\sigma^0)\notin T_P[s+1]$.
We will replace $ext_{\t\xi}$ by $\h{ext}$ which extends
the value of the   $ext_{\t\xi}$ 
so that $\h\rho$ is now in the range of
 $\h{ext} $. Therefore, for $\h\xi\in T_{u-1}\cap \omega^{|\xi|}$,
if  $ext_{\t\xi}(\psi(\h\xi))$
 involves $(\rho^0,\sigma^0)$, then $\psi(\h\xi)$
 has no extension in
 $\h{ext}(A_{\t\xi})\cap (T_P[s+1]\times T_P[s+1])$.
 Such $\h\xi$ will be deleted from $T_{u-1}\cap \omega^{|\xi|}$.
On the other hand, $A_{\t\xi}$ remains.

More specifically,
by Lemma \ref{limitlemcomb0},  there is 
a disjoint preserving extension $\h{ext }:Involve(A_{\t\xi})\rightarrow 3^{|\h\rho|}$
over $(\ch\rho^0,\ch\rho^1)$ such that $\h{ext}(\rho^0) = \h\rho$ and
$\h{ext}(\rho)\succeq ext_{\t\xi}(\rho)$
for all $\rho\in Involve(A_{\t\xi})$.
Let $B\subseteq T_{u-1}\cap \omega^{|\xi|}$ be the set 
of $\h\xi$ such that   $\psi(\h\xi)$ does not involve $(\rho^0,\sigma^0)$.
Clearly $B$ is a proper subset of $T_{u-1}\cap \omega^{|\xi|}$
since $\xi\notin B$.
The point is
\begin{align}\nonumber
\hspace{1cm}B\ne\emptyset\text{ since by Claim \ref{limitclaim3} }
\h{ext}(\psi(B))\text{ is \subsufficient.}
\end{align}

Thus, let $T_u= (T_{u-1}\setminus [\t\xi]^\prec)\cup B$,
we have $T_u$ is a \onestep\ of $T_{u-1}$ (fulfilling the or case).
Let $\h\psi\uhr T_u = \psi\uhr T_u $,
$\phi_u=\phi_{u-1}\uhr T_u$
\footnote{Note that $\Psi^{\h{ext}(\h\psi(\h\xi))}(n) = 
\Psi^{ext_{\t\xi}(\psi(\h\xi))}$ for all $\h\xi\in B$.}.
 Let $f(n,s+1) = \tcp^\smallfrown (T_u,\phi_u)$.
The other part of inductive hypothesis (with respect to 
$\h\psi,\h{ext}$) is easily verified.

\item In case of (\ref{limiteq10}),
let $f(n,s+1)=f(n,s)$.

\end{itemize}

\end{enumerate}

\end{definition}

As we argued in  item (4)
of the definition of $f(n,\cdot)$,
we have $f$ is a $\Gamma_m$-approximation.
By compactness,
$f$ is computable. Since, $C$ is
$\Gamma$-hyperimmune, there are
$n^*,s^*$ such that
\begin{itemize}
\item $C$ diagonal against $f(n^*,s^*)$
and
\item $f(n^*,t) =f(n^*,s^*)$ for all $t\geq s^*$.

\end{itemize}

Suppose $f(n^*,s^*)=((T_0,\phi_0),\cdots,(T_{u-1},\phi_{u-1}))$.
By definition of diagonal against,
there is a $\xi\in \ell(T_{u-1})$
such that $C$ diagonal against $\phi_{u-1}(\xi)$.
Consider $((\h\rho^0,\h\sigma^0),(\h\rho^1,\h\sigma^1))= ext_{\t\xi}(\psi(\xi))$
where $\t\xi$ is the immediate predecessor of $\xi$;
and the $\Pi_1^0$ class $\mcal{T}$ of trees $T$
such that:
\begin{itemize}
\item $T\subseteq T_P\cap ([(\h\rho^0,\h\sigma^0)]^\preceq\cup [(\h\rho^1,\h\sigma^1)]^\preceq)$
 and $(\h\rho^i,\h\sigma^i)$
is \good\ for $[T]$ for all $i\in 2$; and

\item for every $(\t\rho^i,\t\sigma^i)
\in T\cap [(\h\rho^i,\h\sigma^i)]^\preceq$,
if $(\t\rho^0,\t\rho^1)$ does not positively progress on any color $j\in 3$ compared to $(\h\rho^0,\h\rho^1)$,
 then $\Psi^{((\t\rho^0,\t\sigma^0),(\t\rho^1,\t\sigma^1))}(n^*)
  = \phi_{u-1}(\xi)$.

\end{itemize}

Since case $f(n^*,t) =f(n^*,s^*)$ for all $t\geq s^*$,
in particular, case (\ref{limiteq14}) does not occur after 
step $s^*$,
therefore the class $\mcal{T}$ is non empty
(check the definition of sufficient).
Since $\mcal{T}$ is a $\Pi_1^0$ class,
by Lemma \ref{limitlem5}, let $\h T\in \mcal{T}$
be such that $C$ is $\Gamma$-hyperimmune relative to $\h T$.
Since $\psi(\xi)\in A_{\t\xi}\subseteq \Tcross$ and 
$ext_{\t\xi}$ is disjoint preserve extension over $(\ch\rho^0,\ch\rho^1)$,
we have $(\h\rho^0,\h\rho^1)$ does not positively progress on any color $j\in 3$
compared to $(\ch\rho^0,\ch\rho^1)$.
Combine with
$(\h\rho^i,\h\sigma^i)$ being \good\ for $[\h T]$ for all $i\in 2$, we have
$((\h\rho^0,\h\sigma^0),(\h\rho^1,\h\sigma^1), [\h T],\h T)$
 is a condition extending $((\ch\rho^0,\ch\sigma^0),(\ch\rho^1,\ch\sigma^1),P,D)$.
By definition of $\mcal{T}$, $((\h\rho^0,\h\sigma^0),(\h\rho^1,\h\sigma^1), [\h T],\h T)$
 forces the requirement $\mcal{R}_\Psi$.

\end{proof}

Now the rest of the proof follows exactly as that of \ref{limitlem2}.

\end{proof}

\subsection{Yet some other basis theorem}
Note that if two colorings are almost disjoint, then they are Turing equivalent.
Therefore, if we replace ``$Y^0,Y^1$ are not almost disjoint" in 
Lemma \ref{limitlem2} by $Y^0,Y^1$ are not Turing equivalent,
then we arrive at the following question.

\begin{question}\label{limitques1}
Given two incomputable Turing degree $D_0\ngeq_T D_1$,
a non empty $\Pi_1^0$ class $Q\subseteq 2^\omega$,
does there exist a $X\in Q$ such that
$X\ngeq_T D_0$ and $D_0\oplus X\ngeq_T D_1$?

\end{question}

Clearly the difficulty of question \ref{limitques1}
is that when forcing some fact about $D_0\oplus X$,
the condition we use is usually a $\Pi_1^{0, D_0}$
class. But there is no guarantee that a $\Pi_1^{0, D_0}$
class contains a member not computing $D_0$.

\section{Product of infinitely many Ramsey's theorem}
\label{limitsec3}

Note that $(\msf{RT}_2^1)^\omega$ is capable of encoding 
any hyperarithematic degree since it is capable of encoding
fast growing function.

\begin{proposition}
For every function $f\in \omega^\omega$,
there is a $(\msf{RT}_2^1)^\omega$ instance 
$(C_0,C_1,\cdots)$ such that every solution 
of $(C_0,C_1,\cdots)$ compute 
a function $g$ such that $g\geq f$.

\end{proposition}
\begin{proof}
Fix $f\in\omega^\omega$.
Let $C_n\in 2^\omega$ be such that 
$C^{-1}(0) = [0,f(n)]\cap \omega$.
Let $(G_0,G_1,\cdots)$ be a solution 
to $(C_0,C_1,\cdots)$.
Clearly $G_n\subseteq C^{-1}_n(1)$ for all $n\in\omega$.
Thus clearly $(G_0,G_1,\cdots)$
compute a function $g\geq f$.

\end{proof}
On the other hand,
strong
 cone avoidance of non hyperarithmetic degree
  for $(\msf{RT}_2^1)^\omega$
follows from  the following Solovay's theorem \cite{Solovay1978Hyperarithmetically}.

\begin{theorem}[Solovay]
\label{limitth5}
For every non hyperarithmetic  Turing degree $D$,
there exists an infinite set $X$
such that none of the subset of $X$ computes
$D$.

\end{theorem}
  
\begin{proposition}\label{limitth1}
Given a non hyperarithmetic Turing degree $D$,
 every $(\msf{RT}_2^1)^\omega$
instance admit a solution that does not   compute $D$.
\end{proposition}
\begin{proof}
Fix a $(\msf{RT}_2^1)^\omega$ instance $(C_0,C_1,\cdots)$.
Let $X$ be an infinite set as in Solovay's theorem.
Let $G=\{n_0<n_1<\cdots\}$ be an infinite subset of $X$
such that for every $t\in\omega$,
$G\cap [n_t,\infty)$
is monochromatic for $C_t$.
Clearly $G$ computes a solution to $(C_0,C_1,\cdots)$
but $G$ does not compute $D$.

\end{proof}

We wonder if $(\msf{RT}_2^1)^\omega$ is capable of encoding
$\msf{RT}_3^1$.
\begin{question}
Is it true that $\msf{RT}_3^1\leq_{soc} (\msf{RT}_2^1)^\omega$?
\end{question}

\bibliographystyle{amsplain}
\bibliography{F:/6+1/Draft/bibliographylogic}

\end{document}

%% file: sectioncrossconstraint.tex
\section{Complexity of the cross constraint}
\label{limitsec1}

\def\force{\vdash}
\def\nforce{\nvdash}

Let $Q\subseteq 3^\omega\times (2^\omega)^r$
be   such that
$Q$ has full projection on $3^\omega$.
We wonder how complex does $Q$ has to be in order to
satisfy the following cross constraint:
\begin{align}\label{limiteq3}
&\text{for every }(X^0,Y^0),(X^1,Y^1)\in Q,
\text{ if }X^0,X^1\text{ are almost disjoint}, \\ \nonumber
&\text{ then }
Y^0_s,Y^1_s\text{ are almost disjoint for some }s<r.
\end{align}

It turns out that $Q$ can not be analytic.

\begin{proposition}
If $Q$ is $\Sigma_1^1$, then $Q$ does not satisfy
$($\ref{limiteq3}$)$.
\end{proposition}

\begin{proof}

The proof is by constructing
a sequence of \good\ pairs (as a sequence of conditions)  as in
Lemma \ref{limitlem1},
 except
that we redefine ``\good" to adapt to $\Sigma_1^1$ set.

We inductively define a forcing notion
in spirit of Cohen forcing.
For a closed set $\h Q\subseteq \omega^\omega\times \omega^\omega$,
let $\h T$ be the tree defining $\h Q$, i.e., $[\h T]=\h Q$;
for a $\rho,\sigma\in \omega^{<\omega}$,
\begin{itemize}
\item we write
$\rho\force_0 \h Q\cap [(\rho,\sigma)]=\emptyset$
if $(\rho,\sigma)\notin \h T$;

\item for an ordinal $\gamma$, we write
$\rho\force_\gamma \h Q\cap [(\rho,\sigma)]=\emptyset$
if the set
$$\big\{
\h\rho\succeq\rho: \text{ for some }
\h \gamma<\gamma,
\h \rho \force_{\h \gamma}\h Q\cap [(\h\rho,\sigma^\smallfrown n)]=\emptyset
\big\}$$ is dense over $\rho$
for all $n\in\omega$.

\item we write
$\rho\force \h Q\cap [(\rho,\sigma)]=\emptyset$
if $\rho\force_\gamma \h Q\cap [(\rho,\sigma)]=\emptyset$
for some ordinal $\gamma$.
 \end{itemize}

\begin{claim}\label{limitclaim1}
If $\rho\nforce\ \h Q\cap [(\rho,\sigma)]=\emptyset$,
then for some $n\in\omega$, some $\t\rho\succeq\rho$,
we have:
$\h\rho\nforce\ \h Q\cap [(\h\rho,\sigma^\smallfrown n)]=\emptyset$
for all $\h \rho\succeq\t\rho$.
\end{claim}
\begin{proof}
For each $n\in\omega$,
consider the set
$$
A_n=\{
\h\rho\succeq\rho: \h\rho\force \h Q\cap [(\h\rho,\sigma^\smallfrown n)]=\emptyset
\}.
$$
Since $A_n$
is countable, there exists an ordinal $\gamma$ such that
$$A_n= \{
\h\rho\succeq\rho: \h\rho\force_\gamma \h Q\cap [(\h\rho,\sigma^\smallfrown n)]=\emptyset
\}\text{ for all }n.$$

If $A_n$ is dense over $\rho$ for all $n$,
then we have $\rho\force\ \h Q\cap [(\rho,\sigma)]=\emptyset$,
a contradiction.
Therefore, suppose $A_n$ is not dense over $\rho$.
Then there exists $\t\rho\succeq\rho$
such that $A_n\cap [\t\rho]^\preceq=\emptyset$.
Thus $\t\rho,\sigma^\smallfrown n$
is the desired object.

\end{proof}

We think of $Q$ as a subset set of the product space of
$3^\omega$ and $(2^\omega)^r\times \omega^\omega$.
A \emph{\good} triple is a triple
$(\rho,\sigma,\tau)\in 3^{<\omega}\times 2^{<\omega}\times \omega^{<\omega}$
such that
for every $\h \rho\succeq \rho$,
$\h \rho\nforce\ Q\cap [(\h\rho,\sigma,\tau)]=\emptyset$.

Similarly as (\ref{limiteq1}), we have the following:
 \begin{align}\nonumber
 &\text{ for every good triple }
 (\rho,\sigma,\tau),\text{ every }\h\rho\succeq \rho,
 \\ \nonumber
 &\text{ there exists a good triple }
 (\t\rho,\t\sigma,\t\tau)\succ(\h\rho,\sigma,\tau).
 \end{align}

The rest of the proof proceeds exactly as
Lemma \ref{limitlem1}.

\end{proof}

On the other hand,
Johannes Schürz point out that a non principal
ultrafilter on $\omega$ give rise to such a set $Q$.
Recall that a non principal ultrafilter $\mcal{U}$ on $\omega$
is a collection of sets of integers such that:
\begin{itemize}
\item
$\mcal{U}$ is closed under finite intersection (filter),
\item $\mcal{U}$ is upward closed (filter);
\item for every $Z\subseteq \omega$, either $Z\in \mcal{U}$
or $\omega\setminus Z\in \mcal{U}$ (ultrafilter);

\item it is not of form $\{Z\subseteq \omega: n\in Z\}$ for some $n\in\omega$
(non principal).
\end{itemize}
\begin{proposition}[Johannes Schürz]\label{limitprop1}
If there exists a non principal ultrafilter on $\omega$,
then there exists a function $f:3^\omega\rightarrow (2^\omega)^2$
such that for every two almost disjoint\ $X^0,X^1\in 3^\omega$,
$f(X^0)_s,f(X^1)_s$ are almost disjoint\ for some $s<2$.
\end{proposition}
\begin{proof}
Let $\mcal{U}$ be a non principal ultrafilter on $\omega$.
For $X\in 3^\omega$,
let $h(X)= j$ iff $X^{-1}(j)\in \mcal{U}$.
Let $f(X) = (\emptyset,\emptyset),(\emptyset,\omega), (\omega,\emptyset)$
respectively
iff $h(X) = 0,1,2$ respectively
(where $(\emptyset,\emptyset)$ represents
such a $(Y_0,Y_1)\in (2^\omega)^2$ that $Y_0^{-1}(1)=Y_1^{-1}(1)=\emptyset$,
and similarly for $(\omega,\emptyset), (\emptyset,\omega)$).
Note that if $X^0,X^1\in 3^\omega$ are almost disjoint,
then $h(X^0)\ne h(X^1)$, since $\mcal{U}$ is non principal
(which means no finite set is in $\mcal{U}$).
But clearly, if $h(X^0)\ne h(X^1)$,
then $f(X^0)_s,f(X^1)_s$ are almost disjoint for some $s<2$.

\end{proof}

The following result shows that it is consistent with
$\msf{ZFC}$ that ``there exists a $\Pi_1^1$
set $Q\subseteq 3^\omega\times (2^\omega)^3$
with full projection on $3^\omega$
satisfying (\ref{limiteq3})".

\begin{proposition}[Jonathan]\label{limitprop0}
If $\mathbf{V}=\mathbf{L}$, then
there exists a $\Pi_1^1$ set $Q\subseteq 3^\omega\times (2^\omega)^3$
$($a $\Sigma_2^1$ set $Q\subseteq 3^\omega\times (2^\omega)^2$ respectively$)$
with full projection on $3^\omega$ satisfying $($\ref{limiteq3}$)$.
\end{proposition}
\begin{proof}
In $\mathbf{L}$, we
can construct a $\Sigma_2^1$ non principal ultrafilter
$\mcal{U}$
on $\omega$.
Suppose $\mcal{U}$ is the projection of a $\Pi_1^1$ set
$\mcal{V}\subseteq 2^\omega\times 2^\omega$ on the first component.
Let $Y^0,Y^1,Y^2$ denote $(\emptyset,\emptyset), (\emptyset,\omega),(\omega,\emptyset)
\in (2^\omega)^2$
respectively.
Let $Q$ consists of such $(X,Y)\in 3^\omega\times (2^\omega)^3$ that:
\begin{align}\nonumber
 \text{ for some }j<3, Y= (Y^j,\h Y)\text{ and }
(X^{-1}(j), \h Y^{-1}(1))\in \mcal{ V}.
\end{align}

Clearly $Q$ is $\Pi_1^1$ and has full projection on
$3^\omega$.
It's also clear that  $Q$ satisfies (\ref{limiteq3}).

\end{proof}

As demonstrated in Proposition \ref{limitprop0},
a set $Q$ in $3^\omega\times (2^\omega)^2$
is more complex than one in $3^\omega\times (2^\omega)^3$.
We wonder if it is necessarily more complex.
i.e., let $\msf{ECC(r)}$ denote the assertion
``
 there is  a set $Q\subseteq 3^\omega\times (2^\omega)^r$
 satisfying (\ref{limiteq3})"; let
 $\msf{EU(\omega)}$ denote the assertion ``there exists a non principal ultrafilter on $\omega$".
 By Proposition \ref{limitprop1},
 over $\msf{ZF}$, $\msf{EU}(\omega)\rightarrow \msf{ECC}(r)\rightarrow \msf{ECC}(r-1)$
 for all $r>3$.
 \begin{question}
 Are implications (set theoretic) in
 $\msf{EU}(\omega)\rightarrow \msf{ECC}(r)\rightarrow \msf{ECC}(r-1)$
strict?
 \end{question}